\documentclass[final,3p,times]{elsarticle}
\usepackage{amssymb}
\usepackage[usenames, dvipsnames]{color}
\usepackage{amsmath,amscd,amssymb,latexsym}
\usepackage{latexsym}
\usepackage{amsthm}
\usepackage{amsmath}
\newtheorem{thm}{Theorem}[section]
\newtheorem{rem}{Remark}[section]
\newtheorem{lem}{Lemma}[section]
\newtheorem{pro}{Proposition}[section]
\newtheorem{prolet}{Proposition}

\newtheorem{cor}{Corollary}[section]
\newtheorem{definition}{Definition}
\allowdisplaybreaks
\numberwithin{equation}{section}

\journal{}
\begin{document}           
	
	\begin{frontmatter}
		\title{The Nehari manifold for indefinite Kirchhoff problem with Caffarelli-Kohn-Nirenberg type critical growth}
	
	\author[Pawan]{Pawan Kumar Mishra}
	\ead{pawanmishra@mat.ufpb.br}
		\author[Pawan]{Jo\~ao Marcos do \'O \corref{mycorrespondingauthor}}
		\ead{jmbo@pq.cnpq.br}
		\cortext[mycorrespondingauthor]{Corresponding author}
		\address[Pawan]{Department of Mathematics, Federal University of  Para\'iba,\\
			Jo\~ao Pessoa, PB, 58051--900, Brazil}
			\author[costa]{David G. costa}
		\ead{costa@unlv.nevada.edu}
			\address[costa]{Department of Mathematical Sciences,
			University of Nevada,\\ Las Vegas  NV,  89154-4020, USA}

		\begin{abstract}
			In this paper we study the following class of nonlocal {problems} involving Caffarelli-Kohn-Nirenberg type critical growth\\
			\begin{align*}
			L(u)&-\lambda h(x)|x|^{-2(1+a)}u=\mu f(x)|u|^{q-2}u+|x|^{-pb}|u|^{p-2}u\;\;  \text{in } \mathbb R^N,
			\end{align*}
			where
			$h(x)\geq 0$, $f(x)$ is a continuous function which may change sign, $\lambda, \mu$ are positive real parameters and  $1<q<2$, $4< p=2N/[N+2(b-a)-2]$, $0\leq a<b<a+1<N/2$, $N\geq 3$. Here
			$$
			L(u)=-M\left(\int_{\mathbb R^N} |x|^{-2a}|\nabla u|^2dx\right)\mathrm {div}(|x|^{-2a}\nabla u)
			$$
			and the function $M:\mathbb R^+\cup \{0\} \to\mathbb R^+$ is exactly as in the Kirchhoff model, given by $M(t)=\alpha+\beta t$, $\alpha, \beta>0$. Using the idea {of the constrained minimization on} Nehari manifold we show the existence of at least two positive solutions for suitable choices of  $\lambda$ and $\mu$.
				\end{abstract}
		
		
	\end{frontmatter}
	
\section{Introduction}
\noindent In this paper, we are concerned with the existence and multiplicity of positive solutions for the following class of nonlocal problem involving Caffarelli-Kohn-Nirenberg type critical growth\\
\begin{equation}\label{Pc}
L(u)-\lambda h(x)|x|^{-2(1+a)}u=\mu f(x)|u|^{q-2}u+|x|^{-pb}|u|^{p-2}u\;\;  \text{in } \mathbb R^N,
\end{equation}
where
$$
L(u)=-M\left(\int_{\mathbb R^N} |x|^{-2a}|\nabla u|^2dx\right)\mathrm {div}(|x|^{-2a}\nabla u),
$$
is a nonlocal operator involving the Kirchhoff term $M:\mathbb R^+\cup \{0\}\rightarrow\mathbb R^+$ modeled as  $M(t)=\alpha+\beta t$ with $\alpha, \beta>0$. In \eqref{Pc}, $1<q<2$, ${4<p}:=2N/(N+2(b-a)-2)$ is the Caffarelli-Kohn-Nirenberg type critical exponent with  $0\leq a<b<a+1<N/2$ and the parameters $\lambda$ and $\mu$ are positive. Moreover,  $h(x)\geq 0$ and  $f(x)$ satisfies the following assumptions:
\begin{itemize}
	\item [\textbf{(F)}]   $f: \mathbb R^N\to \mathbb R$ is a sign changing, continuous function such that  $f^+=\{x\in \mathbb R^N: f(x)>0\}\neq \emptyset$ and 
	$$
	\|f\|_o:=\left(\int_{\mathbb R^N}|x|^\frac{bpq}{p-q}|f(x)|^\frac{p}{p-q} dx\right)^\frac{p-q}{p}<\infty.
	$$

\end{itemize}
Problems of the type \eqref{Pc} are motivated from the interpolation inequalities
proved by Caffarelli, Kohn and Nirenberg in \cite{MR768824}. Those involving a Caffarelli-Kohn-Nirenberg type nonlinearity have been studied by  many authors in the recent past, see \cite{MR2652606, MR3056960, MR3570122, MR3504429, MR3045175, MR3111855, MR2855595}.\\

Problems of the type \eqref{Pc} (in the case of $a=0$) are related to the stationary analogue of the Kirchhoff type quasilinear hyperbolic  equations such as 
\[
u_{tt}-M \left(\int_{\mathbb R^N} |\nabla u|^2 dx \right)=g(x, t, u),
\]
where $M(t)=\alpha+\beta t$, $\alpha, \beta>0$. It was proposed by Kirchhoff \cite{Kirchhoff} as an extension of the classical D’Alembert’s wave equation for free vibrations of an elastic string. This model incorporates the changes in length of the string occurred during the transverse vibrations. We refer to a servey \cite{Arosio} on this topic. This class of problem received much attention only after Lions \cite{LionsK} proposed an abstract framework to the problem. We cite \cite{Alves1, Arosio2, GF2, Naimen} and references therein for more details.

Problem \eqref{Pc} is called nonlocal because of the presence of the Kirchhoff term which means that \eqref{Pc} is no longer a pointwise identity. This phenomenon causes some mathematical difficulties and makes the problem particularly interesting. {For example, the weak limit of minimizing Palais-Smale sequence may not be a weak solution. The presence of Kirchhoff term requires a compactness result in order to make sure this fact. Also the comparison of energy levels of the problem (in different decompositions of Nehari manifold, see Section 2 below for the definitions) with compactness levels involves some non-trivial estimates.} 

In the case $M=1$, authors in \cite{MR3056960, MR2855595} have addressed  a  similar but subcritical quasilinear elliptic problem  in $\mathbb R^N$ and using the idea of Nehari manifold authors succeeded in showing existence of  multiple solutions. {The results in the present  paper can be considered as the extension of the work of \cite{MR3056960, MR2855595} for the problems involving critical growth as well as a Kirchhoff term. Moreoever, the results also can be seen as the extension for a nonlocal Kirchhoff problem with a sublinear perturbation of the work in \cite{MR2652606}, where authors have considered the problem \eqref{Pc} for $M\equiv 1$ and $\mu=0$. }

In the case $f=1$ and $\lambda=0$, authors in \cite{MR3570122} studied the quasilinear situation on bounded domains involving $p$-sublinear and $p$-superlinear terms using the Krasnoselskii genus in a variational framework and,  under some {suitable} assumptions on the parameters the existence of infinitely many solutions was established. Further, in \cite{MR3504429}, authors have complemented those results by studying the $p$-linear situation and showing existence results for any $\mu>0$ even in the $p$-superlinear case.

The range of the parameter $\lambda$ will be determined by the principal eigenvalue of the following eigenvalue problem
\begin{equation}\label{Ec}
-\mathrm {div}(|x|^{-2a}\nabla u)=\lambda h(x)|x|^{-2(1+a)}u\;\;  \text{in } \mathbb R^N\setminus\{0\}.
\end{equation}
In the present paper, we aim to obtain existence of two positive solutions for \eqref{Pc} when $\lambda$ is in a suitable (scaled) neighborhood of the principal eigenvalue of \eqref{Ec} and for sufficient small values of $\mu>0$. We make use of constrained minimization technique combined with the concentration-compactness principle of P.-L. Lions \cite{MR850686} to find the minimizers of the associated energy functional.\\
	{
\noindent The paper is organized as follows. In section 2 we discuss the variational formulation of the problem and state the main result of the paper. In section 3 we introduce the associated Nehari manifold and related  fibering maps. In section 4, we extract Palais-Smale sequences out of Nehari decompositions. Compactness results are studied in Section 5. Section 6 and Section 7 are dedicated to prove  the existence of first and second solutions respectively.}

\section{{The  variational setting}}
For any $r\in [1, \infty)$ and $c\geq 0$, we denote by $L_c^r(\mathbb R^N):=L^r( \mathbb R^N, |x|^{-rc}dx)$ the Banach space of measurable functions on $\mathbb R^N$ whose $r^{th}$ power is Lebesgue integrable with respect to the measure $|x|^{-rc}dx$, endowed with the norm
$$
\|u\|_{L_c^r}:=\left(\int_{\mathbb R^N}|x|^{-rc}|u|^r dx\right)^\frac1r.
$$
The following Caffarelli-Kohn-Nirenberg inequality will be used in what follows:
\begin{equation}\label{CKN}
{S}\left(\int_{\mathbb R^N}|x|^{-pb}|u|^p dx\right)^\frac2p\leq \int_{\mathbb R^N}|x|^{-2a}|\nabla u|^2\,dx\ \quad \text{for all}\;\; u\in D_a^{1, 2}(\mathbb R^N),
\end{equation}
where $D_a^{1, 2}(\mathbb R^N)$ is the completion of $C_c^\infty(\mathbb R^N)$ with respect to the norm
\begin{equation}\label{d12norm}
\|u\|:=\left(\int_{\mathbb R^N}|x|^{-2a}|\nabla u|^2 dx\right)^\frac12\,.
\end{equation}
{and $S$ is the best Sobolev constant of the corresponding continuous embedding of  $D^{1,2}_a(\mathbb R^N)$ into $L^p_b(\mathbb R^N)$.} {Next, we state the following proposition about the eigen value problem \eqref{Ec}. The proof is ommited as it is similar to Proposition 1.1 of \cite{MR2652606} in the case $p=2$.}
{
\begin{prolet}
Suppose $0\not \equiv h\geq 0$ satisfies
\begin{itemize}
\item [\textbf{(H)}] $h\in L^{{N/{p_0}}}_{p_0}(\mathbb R^N)\cap L^{{N/{p_0}}+\theta}_{\rm{loc}}(\mathbb R^N\setminus \{0\})$ for  some $\theta>0$ and $p_0=2-2(b-a)$.
\end{itemize}

 Then the nonlinear eigenvalue problem \eqref{Ec} has a principal eigenvalue $\lambda_1=\lambda_1(h)>0$ which is simple. Moreover, a corresponding eigenfunction $\phi_1$ belongs to the space $D_a^{1, 2}(\mathbb R^N)$ and can be taken to be positive in the sense that $\phi_1>0$ a.e. in $\mathbb R^N\setminus \{0\}$.
\end{prolet}
}
 The principal eigenvalue of \eqref{Ec} is given by
$$
\frac{1}{\lambda_1(h)}\ =\displaystyle \sup_{u\in D_a^{1,2}(\mathbb R^N)}\displaystyle\frac{\int_{\mathbb R^N}h(x)|x|^{-2(1+a)}u^2dx}{\int_{\mathbb R^N} |x|^{-2a}|\nabla u|^2dx}\,,
$$
so that
\begin{equation}\label{EVPm}
\int_{\mathbb R^N} \left(|x|^{-2a}|\nabla u|^2-\lambda h(x)|x|^{-2(1+a)}u^2\right)dx>0
\end{equation}
for every $u\in  D_a^{1,2}(\mathbb R^N)$ and $0<\lambda<\lambda_1(h)$. It can be shown that, for every $0\leq \lambda<\lambda_1(h)$, there exists $\delta(\lambda)>0$ such that
\begin{align}\label{egcm}
\int_{\mathbb R^N} \left(|x|^{-2a}|\nabla u|^2-\lambda h(x)|x|^{-2(1+a)}u^2\right)dx>\delta(\lambda)\int_{\mathbb R^N} |x|^{-2a}|\nabla u|^2dx
\end{align}
for all $u\in  D_a^{1, 2}(\mathbb R^N)$.\\

{
\begin{definition}\label{defws}
	 A function $u\in  D_a^{1, 2}(\mathbb R^N)$ is said to be a weak solution of the problem if, for every $v\in  D_a^{1, 2}(\mathbb R^N)$, the following holds
	\begin{equation*}
	M\left(\|u\|^2\right)\langle u, v\rangle -{\lambda}\displaystyle \int_{\mathbb R^N}h(x)|x|^{-2(1+a)}u vdx-{\mu}\displaystyle\int_{\mathbb R^N}f(x)|u|^{q-2}uvdx-\displaystyle\int_{\mathbb R^N}|x|^{-pb}|u|^{p-2}uvdx=0,
	\end{equation*}
	where $\langle \cdot , \cdot  \rangle$ is a inner product on $D_a^{1, 2}(\mathbb R^N)$ which induces the norm $\|\cdot\|$, defined in \eqref{d12norm}.
\end{definition}
}
In the sense of Definition \ref{defws}, we  state the main result of the paper about the existence of weak solutions as follows:
\begin{thm}\label{22mht1}
	Let $h(x)$, $f(x)$ satisfy $\textbf{(H)}$ and $\textbf{(F)}$ respectively and $M(t)=\alpha+\beta t$. Assume $\lambda \in (0, \alpha \lambda_1(h))>0$, where  $\lambda_1(h)$ is the principle eigenvalue of \eqref{Ec}. Then
	\begin{enumerate}
		\item  there exists $\mu_0>0$ such that problem \eqref{Pc} has at least one positive solution with negative energy for all $\alpha>0$ and $\beta>0$.
		\item there exists $\mu_{00}>0$ such that, for all  $0<\mu<\mu_{00}\leq\mu_0$,  problem \eqref{Pc} has at least two positive solutions for all $\alpha>0$ and  $\beta>0$ sufficiently small.
	\end{enumerate}
\end{thm}
The energy functional associated with the problem \eqref{Pc} is 
\begin{align*}
J_{\lambda, \mu}(u)=\frac12 \widehat M\left(\|u\|^2\right)-\frac{\lambda}{2}\displaystyle \int_{\mathbb R^N}h(x)|x|^{-2(1+a)}u^2dx-\frac{\mu}{q}\displaystyle\int_{\mathbb R^N}f(x)|u|^qdx-\frac{1}{p}\displaystyle\int_{\mathbb R^N}|x|^{-pb}|u|^pdx,
\end{align*}
where  $\widehat M(t)=\displaystyle\int_0^tM(s)ds$ {is the primitive of } $M$. { Under the light of assumption $\textbf{(F)}$, inequality \eqref{CKN} and \eqref{EVPm} functional $J_{\lambda, \mu}$ is well defined and  is of class $C^1$ on $D_a^{1,2}(\mathbb R^N)$. It is easy to see that the critical points of the functional $J_{\lambda, \mu}$ corresponds to the weak solution of the problem \eqref{Pc}, in the sense of Definition \ref{defws}.}

\section{{The Nehari manifold and fibering map analysis}}
The energy functional $J_{\lambda,  \mu}$ is not bounded below on $D_a^{1,2}(\mathbb R^N)$. Therefore, in order to study the problem \eqref{Pc} through minimization argument we adopt a well explored idea of constrained minimization in the literature,  popularly known as Nehari minimization technique. The  Nehari set is  defined as follows:
$$
N_{\lambda, \mu}=\left\{u\in D_a^{1,2}(\mathbb R^N)\setminus\{0\}\vert \langle J^\prime_{\lambda, \mu}(u), u\rangle_*=0\right\},
$$
where $\langle\, ,\rangle_*$ is the duality between the dual space of $D_a^{1,2}(\mathbb R^N)$ and $D_a^{1,2}(\mathbb R^N)$. Thus  $u\in N_{\lambda, \mu}$ if and only if
\begin{align*}
M\left(\|u\|^2\right)\|u\|^2-{\lambda}\displaystyle \int_{\mathbb R^N}h(x)|x|^{-2(1+a)}u^2dx-{\mu}\displaystyle\int_{\mathbb R^N}f(x)|u|^qdx-\displaystyle\int_{\mathbb R^N}|x|^{-pb}|u|^pdx=0.
\end{align*}
The fact that the energy functional is bounded below on this Nehari subset of $D_a^{1,2}(\mathbb R^N)$ can be seen in the following lemma.
\begin{lem}\label{cbbi}
The energy functional $J_{\lambda, \mu}$ is coercive and bounded below on  ${N}_{\lambda, \mu}$. 

\end{lem}
\begin{proof}
For $u\in {N}_{\lambda, \mu},$ we have
\begin{align*}
J_{\lambda, \mu}(u) &= \left(\frac12-\frac{1}{p}\right)\left(\alpha\|u\|^{2}-\lambda\int_{\mathbb R^N}h(x)|x|^{-2(1+a)}|u|^2dx\right)+\left(\frac14-\frac{1}{{p}}\right)\beta \|u\|^4- \mu \left(\frac{1}{q}-\frac{1}{{p}}\right)\int_{\mathbb R^N} f(x)|u|^{q}dx,\\
&\geq \left(\frac12-\frac{1}{{p}}\right)\left(\alpha\|u\|^{2}-\lambda\int_{\mathbb R^N}h(x)|x|^{-2(1+a)}|u|^2dx\right) - \mu \left(\frac{1}{q}-\frac{1}{{p}}\right) S^{\frac{-q}{2}}\|f\|_o\|u\|^q\\
&\geq\left(\frac12-\frac{1}{{p}}\right)\alpha \delta(\lambda)\|u\|^2- \mu \left(\frac{1}{q}-\frac{1}{{p}}\right) S^{\frac{-q}{2}}\|f\|_o\|u\|^q.
\end{align*}
As $q<2$, it is easy to see that $J_{\lambda, \mu}(u)\to +\infty$ as $\|u\|\to \infty$. Hence $J_{\lambda, \mu}$ is coercive. Now  define
$$
G(t) =\left(\frac12-\frac{1}{{p}}\right)\alpha \delta(\lambda)t^{\frac{2}{q}} - \mu \left(\frac{1}{q}-\frac{1}{p}\right) S^{\frac{-q}{2}}\|f\|_ot\,,
$$
then $G(t)$ attains its minimum at
$$
t=\left(\frac{\mu(p-q)\|f\|_o  S^{\frac{-q}{2}}}{(p-2)\alpha \delta(\lambda)}\right)^\frac{q}{2-q}.
$$
Therefore $J_{\lambda, \mu}(u)\geq-C\mu^\frac{2}{2-q}$
for all $u\in {N}_{\lambda, \mu}$, and some constant $C>0$. Hence $J_{\lambda, \mu}$ is bounded from below.
\end{proof}

Now, to study the structure of Nehari set, we define the fibering map $\Phi_u: \mathbb{R}^+\rightarrow \mathbb{R}$, for every fixed $u\in D_a^{1,2}(\mathbb R^N)$,  as $\Phi_u(t)=J_{\lambda, \mu}(tu)$. Now differentiating $\Phi_u(t)$ with respect to $t$ and using $M(t)=\alpha+\beta t$, we get 
\begin{align*}
	\Phi_u^{\prime}(t)&=t\left(\alpha \|u\|^2 -\lambda\int_{\mathbb R^N}h(x)|x|^{-2(1+a)}|u|^2dx\right)+t^3\beta \|u\|^4
	-t^{q-1}\mu\int_{\mathbb R^N}f(x)|u|^q dx-t^{p-1}\int_{\mathbb R^N}|x|^{-pb}|u|^{p} dx
\end{align*}

Therefore, $tu\in  N_{\lambda, \mu}$ if and only if $\Phi_u^{\prime}(t)=0$. In particular, $u\in N_{\lambda, \mu}$ if and only if $t=1$ is a critical point of $\Phi_u(t)$, that is, 
\begin{align}\label{phider}
\alpha\|u\|^2-\lambda\int_{\mathbb R^N}h(x)|x|^{-2(1+a)}|u|^2dx+\beta \|u\|^4
-\mu\int_{\mathbb R^N}f(x)|u|^q dx-\int_{\mathbb R^N}|x|^{-pb}|u|^{p} dx&=0\,.
\end{align}
Moreover,
\begin{align}\label{phidder}
\Phi^{\prime\prime}_u(1)&=\alpha\|u\|^2-\lambda\int_{\mathbb R^N}h(x)|x|^{-2(1+a)}|u|^2dx+3\beta\|u\|^4
-(q-1)\mu\int_{\mathbb R^N}f(x)|u|^q dx-{(p-1)}\int_{\mathbb R^N} |x|^{-pb}|u|^{p} dx\,.
\end{align}
Now, we split $N_{\lambda, \mu}$ into three parts based on the classification of $t=1$ as a  local minima, local maxima and saddle points of $\Phi_u(t)$ as follows:
\begin{equation}
\begin{aligned}\label{Nsubset}
N_{\lambda,\mu}^{+}&=\{u\in N_{\lambda, \mu}\;|\;\Phi^{\prime\prime}_u(1)> 0\}, \\
N_{\lambda, \mu}^{-}&=\{u\in  N_{\lambda, \mu}\;|\;\Phi^{\prime\prime}_u(1)< 0\}, \\
 N_{\lambda,\mu}^0&=\{u\in N_{\lambda,\mu}\;|\;\Phi^{\prime\prime}_u(1)=0\}.
\end{aligned}
\end{equation}
Using the following Lemma together with the implicit function theorem one can show that the Nehari set defined above  is a $C^1$ manifold of co-dimension 1. Let us denote
\begin{equation}\label{muone}
\mu_1={\left(\frac{p-2}{\|f\|_o}\right)\left(\frac{2-q}{p-q}\right)^\frac{2-q}{p-2}
(\alpha \delta(\lambda)S)^\frac{p-q}{p-2}}
\end{equation}
where $S$ is defined in \eqref{CKN} and $\beta$ is in the Kirchhoff term.
\begin{lem}\label{l22.2i}
${N}_{\lambda, \mu}^{0} = \emptyset \;\textrm{for all} \;\mu \in (0, \mu_{1})$ and $\lambda \in (0, \alpha \lambda_1(h))$.
\end{lem}
\begin{proof}
We consider the following two cases.
\\
\smallskip
\textbf{Case 1:}  In this case we show that if $u\in {N}_{\lambda, \mu}$ such that $\displaystyle \int_{\mathbb R^N} f(x)|u|^{q}dx=0$ then $u\notin N_{\lambda, \mu}^0.$ To prove this fact, we need to show that $\Phi^{\prime\prime}_u(1)\neq 0$. Let us compute $\Phi^{\prime\prime}_u(1)$ using \eqref{phidder} and \eqref{phider} as follows:
\begin{align*}
    \Phi^{\prime\prime}_u(1)&=\alpha\|u\|^2+3\beta\|u\|^4-\lambda\int_{\mathbb R^N}h(x)|x|^{-2(1+a)}|u|^2dx-{(p-1)}\int_{\mathbb R^N} |x|^{-pb}|u|^{p} dx\\
    &= (2-p)\alpha \|u\|^{2} + (4-p)\beta\|u\|^{4}-(2-p)\lambda\displaystyle\int_{\mathbb R^N}h(x)|x|^{-2(1+a)}|u|^2dx\\
    &=(2-p)\left(\alpha \|u\|^{2}-\lambda\displaystyle\int_{\mathbb R^N}h(x)|x|^{-2(1+a)}|u|^2dx\right)+(4-p)\beta\|u\|^{4}<0
\end{align*}
for all $\lambda<\alpha {\lambda_1(h)}$ and $p>4$, which implies $u \notin {N}_{\lambda, \mu}^{0}$\,.
\\
\smallskip
 \textbf{Case 2:} In this case we show that if $u \in {N}_{\lambda, \mu}$ such that  $\displaystyle \int_{\mathbb R^N}f(x) |u|^{q}dx \neq 0$ then $u\notin N_{\lambda, \mu}^0$ if $\lambda \in (0, \alpha \lambda_1(h))$ and $\mu \in (0, \mu_{1})$.  We prove this fact by a contradiction argument. 
Suppose $u \in {N}_{\lambda, \mu}^{0}$, that is, $\Phi^{\prime\prime}_u(1)=0$. Then, from \eqref{phidder}, we have following two conclusions, (eliminating  $\int_{\mathbb R^N} f(x)|u|^q dx$ and $\int_{\mathbb R^N} |x|^{-pb}|u|^p dx$ from \eqref{phidder} using  \eqref{phider}, respectively)
\begin{align}\label{2.11}
    (2-q)\left(\alpha \|u\|^{2}-\lambda\int_{\mathbb R^N}h(x)|x|^{-2(1+a)}|u|^2dx\right)+(4-q)\beta\|u\|^{4} = (p-q)\int_{\mathbb R^N}|x|^{-pb}|u|^{p} dx,
    \end{align}
    \begin{align}\label{2.12}
    (p-2) \left(\alpha \|u\|^{2}-\lambda\int_{\mathbb R^N}h(x)|x|^{-2(1+a)}|u|^2dx\right)
    +(p-4)\beta\|u\|^4 = (p-q)\mu \displaystyle \int_{\mathbb R^N}f(x) |u|^{q}dx.
\end{align}
Now, define $E_{\lambda, \mu}: N_{\lambda, \mu}\rightarrow \mathbb R$ as
\begin{equation*}
    E_{\lambda, \mu}(u) = \frac{(p-2) \left(\alpha \|u\|^{2}-\lambda\int_{\mathbb R^N}h(x)|x|^{-2(1+a)}|u|^2dx\right)+(p-4)\beta \|u\|^4}{p-q} - \mu \displaystyle \int_{\mathbb R^N}f(x) |u|^{q}dx,
\end{equation*}
so that, from equation \eqref{2.12}, $E_{\lambda, \mu}(u) = 0$ for all  $u \in \mathcal{N}_{\lambda, \mu}^{0}.$\\
And, for $0<\lambda<\alpha \lambda_1(h)$,
\begin{align*}
  E_{\lambda, \mu}(u)& \geq  {\left(\frac{p-2}{p-q}\right)\alpha\delta(\lambda)\|u\|^2 }- \mu \|f\|_o S^{-q/2}\|u\|^q,\\
  & \geq  \|u\|^q\left({\left(\frac{{p}-2}{p-q}\right)\alpha\delta(\lambda)\|u\|^{2-q}} - \mu \|f\|_o S^{-q/2}\right).\\
\end{align*}

Next, from equation \eqref{2.11} and \eqref{CKN} with $0<\lambda<\alpha \lambda_1(h)$, we get
\begin{equation}\label{2.13}
  \|u\| \geq{ \left(\frac{\alpha\delta(\lambda)(2-q) S^{p/2}}{p-q}\right)^{\frac{1}{p-2}}.}
\end{equation}
Then, using equation \eqref{2.13} and $\mu\in (0, \mu_1)$ gives $E_{\lambda, \mu}(u)>0$ for all $u \in {N}_{\lambda, \mu}^{0},$ which is a contradiction.
\end{proof}
In every constrained minimization approach the ultimate goal is to show that the minimizer (or a critical point) of the functional obtained under the applied constraint is actually a minimizer (or a critical point) of the functional. The following Lemma shows precisely this fact in context of the Nehari manifold.
\begin{lem}
Let $u$ be a local minimizer for $J_{\lambda, \mu}$ in any of the subsets of $N_{\lambda, \mu}$ defined in \eqref{Nsubset}. If $u\notin N_{\lambda, \mu}^{0}$ then  $u$ is a critical point of $J_{\lambda, \mu}$.
\end{lem}
\begin{proof}
Let $u$ be a local minimizer for $J_{\lambda, \mu}$ in any of the
subsets of $N_{\lambda, \mu}$ defined in \eqref{Nsubset}. Then, in any case $u$ is a
minimizer for $J_{\lambda, \mu}$ under the constraint $I_{\lambda, \mu}(u):=\langle
J_{\lambda, \mu}^{\prime}(u),u\rangle_* =0$. Hence, by the theory of Lagrange
multipliers, there exists $\eta \in \mathbb R$ such that $ J_{\lambda, \mu}^{\prime}(u)= \eta I_{\lambda, \mu}^{\prime}(u)$. Thus $0=\langle
J_{\lambda, \mu}^{\prime}(u),u\rangle_*= \eta\;\langle I_{\lambda, \mu}^{\prime}(u),u\rangle_* = \eta
\Phi_{u}^{\prime\prime}(1)$. Since $\Phi_{u}^{\prime\prime}(1) \ne 0$ as $u\notin N_{\lambda, \mu}^{0}$,  implies $\eta=0$ which proves the Lemma.
\end{proof}
Next, we denote
\begin{equation*}
\begin{aligned}
F^+ &=\left\{u\in  D_a^{1,2}(\mathbb R^N): \displaystyle \int_{\mathbb R^N} f(x)|u|^qdx> 0\right\}\,,\\
F^- &=\left\{u\in  D_a^{1,2}(\mathbb R^N): \displaystyle \int_{\mathbb R^N} f(x)|u|^qdx\leq 0\right\}.
\end{aligned}
\end{equation*}
\noindent As we know that the decompositions $N_{\lambda, \mu}^\pm$ are characterised through the critical points of fibering maps being local maxima or local minima, it is gainful to study the behavior of these maps.  We discuss the behaviour of these maps according to the sign changing behavior of the integral $\displaystyle \int_{\mathbb R^N} f(x)|u|^qdx$ in the following cases. {First we denote
\begin{equation}\label{mutwo}
\mu_2=\left(\frac{p-q-1}{\|f\|_o}\right)\left(\frac{(2-q)\alpha\delta(\lambda)S}{p-q}\right)^\frac{p-q}{p-2}
\end{equation}}
\medskip
\textbf{Case 1: $u\in F^+$}.\\
Define $\psi_u:\mathbb{R}^+\rightarrow \mathbb{R}$ ( by seperating the sublinear term from the equation $\Phi_u^\prime(t)=0$) as
\begin{align}\label{phidef}
\psi_u(t)= t^{2-q}\left(\alpha\|u\|^2-\lambda\int_{\mathbb R^N}h(x)|x|^{-2(1+a)}|u|^2dx\right)+\beta t^{4-q}\|u\|^{4}-t^{p-q}\int_{\mathbb R^N} |x|^{-pb}|u|^{p} dx\,
\end{align}
and observe that $tu\in N_{\lambda, \mu}$ if and only if
\[
\psi_u(t)= \mu \displaystyle \int_{\mathbb R^N} f(x)|u|^qdx\,.
\]
From \eqref{phidef}, we have
\begin{align}\label{siutd}
\psi_u^{\prime}(t)=(2-q)t^{1-q}&\left(\alpha\|u\|^2-\lambda\int_{\mathbb R^N}h(x)|x|^{-2(1+a)}|u|^2dx\right)
+ (4-q)\beta t^{3-q}\|u\|^{4}-(p-q)t^{p-q-1}\int_{\mathbb R^N} |x|^{-pb}|u|^{p} dx\,,
\end{align}
and from \eqref{siutd} it is easy to see that $\psi_{u}^\prime (t) \rightarrow - \infty$ as $t \rightarrow \infty$  and $\lim_{t\rightarrow 0^{+}}\psi'_{u}(t)>0$. Moreover from $\lim_{t\rightarrow \infty} \psi_{u}(t)=-\infty$ and $\psi_u(0)=0$. {Furthermore, $\psi^{\prime\prime}(t)\geq 0$ implies that $\psi^\prime (t)\leq 0$ which implies  the existence of a unique} $ t_{\max}= t_{\max}(u)>0$ such that $ \psi_{u}'(t_{\max})=0$ and $\psi_u(t)$ is increasing on
$(0, t_{\max})$ and decreasing on $(t_{\max}, \infty)$. {On estimating \eqref{siutd}  at $t=t_{\max}$, for $\lambda\in (0, \alpha \lambda_1(h))$, we get
\begin{align*}
    (p-q)t^{p-q-1}_{\max}\int_{\mathbb R^N} |x|^{-pb}|u|^{p}dx&=(2-q)t_{\max}^{1-q}\left(\alpha\|u\|^{2}-\lambda\int_{\mathbb R^N}h(x)|x|^{-2(1+a)}|u|^2dx\right)+ (4-q)\beta t_{\max}^{3-q}\|u\|^{4}\\
    &\geq (2-q)t_{\max}^{1-q}\alpha \delta(\lambda)\|u\|^2.
\end{align*}
Now using \eqref{CKN}, we get 
\begin{align*}
(p-q)t^{p-q-1}_{\max}S^\frac{-p}{2}\|u\|^p\geq  (p-q)t^{p-q-1}_{\max}\int_{\mathbb R^N} |x|^{-pb}|u|^{p}dx\geq (2-q)t_{\max}^{1-q}\alpha \delta(\lambda)\|u\|^2
\end{align*}
which implies
\begin{equation}\label{eneq}
t_{\max} \geq \frac{1}{\|u\|} \left(\frac{(2-q)\alpha\delta(\lambda) S^{p/2}}{(p-q)}\right)^{\frac{1}{p-2}}:=T_1>0\,.
\end{equation}
Using the inequality {\eqref{eneq}} together with the fact that $\psi_u(t)$ is increasing in $(0, t_{\max})$, we have
\begin{align*}
    \psi_{u}(t_{\max}) & \geq \psi_u(T_1)\geq (2-q)\alpha \delta(\lambda)T_1^{2-q}\|u\|^2-T_1^{p-q}\int_{\mathbb R^N} |x|^{-pb}|u|^{p}dx \\
    &\geq (2-q)\alpha \delta(\lambda)T_1^{2-q}\|u\|^2-T_1^{p-q}S^\frac{-p}{2}\|u\|^p\\
    & \ge   (p-q-1)S^\frac{-p}{2}\|u\|^q \left(\frac{(2-q)\alpha\delta(\lambda)S^\frac{p}{2}}{p-q}\right)^\frac{p-q}{p-2}>0.
    \end{align*}
}

Then, if $\mu<\mu_2$, there exists unique $t^+ = t^+(u) < t_{\max}$ and $t^- = t^-(u) > t_{\max}$ such that
\[
\psi_{u}(t^+) = \mu \displaystyle \int_{\mathbb R^N}{f(x)|u|^{q}}dx = \psi_{u}(t^-),
\]
hence  $t^+u, t^-u \in N_{\lambda, \mu}.$ Also $\psi'_{u}(t^+) > 0$ and $\psi_{u}'(t^-) < 0$.  Now, using the relation $\Phi_{tu}^{\prime\prime}(1)=t^{1+q}\psi_u^{\prime}(t)$  for $tu\in N_{\lambda, \mu}$, obtained from \eqref{phidder} and \eqref{siutd},
we get $\Phi_{t^+u}^{\prime\prime}(1)>0$ and $\Phi_{t^-u}^{\prime\prime}(1)<0$, which implies $t^+u \in {N}^{+}_{\lambda, \mu}$ and $t^-u \in {N}^{-}_{\lambda, \mu}.$
\\
\medskip
 \textbf{Case 2: $u\in F^-$}.\\
Since $\psi_{u}(t) \rightarrow -\infty$ as $t \rightarrow \infty$,  for every $u\in F^-$ and for all $\mu>0$, there exists $t^*>t_{\max}$ such that
\[
\psi_{u}(t^*) = \mu \displaystyle \int_{\mathbb R^N}{f(x)|u|^{q}}dx\,,
\]
which implies $t^*u\in {N}_{\lambda, \mu}$. Moreover, as $\psi_u^\prime(t)<0$ for $t>t_{\max}$, we get  $\psi_u^\prime(t^*)<0$. Therefore, using again the relation $\Phi_{tu}^{\prime\prime}(1)=t^{q+1}\psi_u^{\prime}(t)$, we get $\Phi_{t^*u}^{\prime\prime}(1)<0$. Hence
$t^*u\in N_{\lambda, \mu}^-$.
\\
\medskip

Next, let us define
\begin{equation*}
\begin{aligned}
\theta _{\lambda, \mu}&=\inf\{ J_{\lambda, \mu}(u): u\in N_{\lambda, \mu}\},\\
\theta _{\lambda, \mu}^+&=\inf\{ J_{\lambda, \mu}(u): u\in N_{\lambda, \mu}^+\},\\
 \theta _{\lambda, \mu}^-&=\inf\{ J_{\lambda, \mu}(u): u\in  N_{\lambda, \mu}^-\}.
\end{aligned}
\end{equation*}
Then we have the following Lemma:
\begin{lem}\label{3a}
{For $\lambda\in (0, \alpha \lambda_1(h))$  we have $\theta_{\lambda, \mu}\leq \theta_{\lambda, \mu}^+ <0.$}
\end{lem}
\begin{proof}
Let $u\in D^{1,2}_a(\mathbb R^N)$. {Since $f^+\neq \emptyset$,} it can be assumed  without loss of generality that $u\in F^+$. Then, as discussed in $\textbf {Case 1}$ above, there exists a unique $t^+>0$ such that {$t^+u\in N_{\lambda, \mu}\cap N_{\lambda, \mu}^+$}. Denoting $t^+u=v$, we have from \eqref{phider}
\begin{equation}\label{pglad}
{\mu}\int_{\mathbb R^N}f(x)|v|^q dx=\left(\alpha\|v\|^2-\lambda\int_{\mathbb R^N}h(x)|x|^{-2(1+a)}|v|^2dx\right)+{\beta}\|v\|^4-\int_{\mathbb R^N}|x|^{-pb}|v|^{p} dx\,.
\end{equation}
Now, eliminating the term $\displaystyle\int_{\mathbb R^N}f(x)|v|^q dx$ from $J_{\lambda, \mu}$ by using \eqref{pglad}, we get

\begin{equation}\label{jstar}
\begin{aligned}
 J_{\lambda, \mu}(v) &= \left(\frac{1}{2}-\frac{1}{q}\right)\left(\alpha \|v\|^{2} -\lambda\int_{\mathbb R^N}h(x)|x|^{-2(1+a)}|v|^2dx\right)+ \left(\frac{1}{4}-\frac{1}{q}\right)\beta\| v\|^{4}+ \left(\frac{1}{q}-\frac{1}{p}\right) \int_{\mathbb R^N}|x|^{-pb}|v|^{p} dx.
\end{aligned}
\end{equation}
Moreover, as $v\in N_{\lambda, \mu}^+$, from the definition in \eqref{Nsubset} we get
\[
{(p-1)}\int_{\mathbb R^N} |x|^{-pb}|v|^{p} dx<\alpha\|v\|^2-\lambda\int_{\mathbb R^N}h(x)|x|^{-2(1+a)}|v|^2dx
+3\beta\|v\|^4-(q-1)\mu\int_{\mathbb R^N}f(x)|v|^q dx\,.
\]
Next, using again \eqref{pglad} to eliminate the term $\displaystyle\int_{\mathbb R^N}f(x)|v|^q dx$ in the above inequality , we get
\begin{equation*}
  \int_{\mathbb R^N}|x|^{-pb}|v|^{p} dx \leq \left(\frac{2-q}{p-q}\right)\left(\alpha\|v\|^{2}-\lambda\int_{\mathbb R^N}h(x)|x|^{-2(1+a)}|v|^2dx\right)+ \left(\frac{4-q}{p-q}\right)\beta\|v\|^{4}\,.
\end{equation*}
{Substituting the above inequality in \eqref{jstar}, we have }
\begin{align*}
J_{\lambda, \mu}(v) &\leq -\frac{(2-q)(p-2)}{2pq} \left(\alpha\|v\|^{2}-\lambda\int_{\mathbb R^N}h(x)|x|^{-2(1+a)}|v|^2dx\right)-\frac{(4-q)(p-4)}{4pq} \beta \|v\|^{4}\,,\\
&\leq-\frac{(2-q)(p-2)}{2pq} \alpha \delta(\lambda)\mathrm{C}\,,
\end{align*}
where $\mathrm{C}=\|v\|^2\,.$ This implies $ \theta_{\lambda, \mu}^+<0$.
\end{proof}
{The following Lemma helps in showing that the set  $N_{\lambda,\mu}^-$ is closed in the $D^{1,2}_a(\mathbb R^N)$ topology.}
\begin{lem}\label{nlwi}
There exists $\rho>0$ such that $\|u\|\geq \rho$ for all $u\in N_{\lambda,\mu}^-$.
\end{lem}
\begin{proof}
Let $u\in N_{\lambda,\mu}^-$. Then, from \eqref{Nsubset} we get
\begin{align*}
\alpha\|u\|^2-\lambda\int_{\mathbb R^N}h(x)|x|^{-2(1+a)}|u|^2dx+3\beta\|u\|^4-\mu(q-1)\int_{\mathbb R^N}f(x)|u|^qdx<(p-1)\int_{\mathbb R^N}|x|^{-pb}|u|^{p} dx.
\end{align*}
Now, eliminating the term $\displaystyle\int_{\mathbb R^N}f(x)|u|^q dx$ using the fact that $u\in N_{\lambda, \mu}$ (recall \eqref{phider}), we get
\begin{equation}\label{pglagd}
(2-q)\left(\alpha \|u\|^2-\lambda\int_{\mathbb R^N}h(x)|x|^{-2(1+a)}|u|^2dx\right)+(4-q)\beta\|u\|^4
 <(p-q)\int_{\mathbb R^N}|x|^{-pb}|u|^{p} dx\,.
\end{equation}
Also, the first term above in \eqref{pglagd} can be estimated from below by using \eqref{egcm} for $\lambda<\alpha \lambda_1(h)$, as
\begin{equation}\label{pgla1d}
(2-q)\alpha\delta(\lambda)\|u\|^2\leq (2-q)\left(\alpha \|u\|^2-\lambda\int_{\mathbb R^N}h(x)|x|^{-2(1+a)}|u|^2dx\right)\,.
\end{equation}
So, by combining \eqref{pglagd} and \eqref{pgla1d}, we get
\[
(2-q)\alpha\delta(\lambda)\|u\|^2
+(4-q)\beta\|u\|^4
 <(p-q)\int_{\mathbb R^N}|x|^{-pb}|u|^{p} dx\,.
\]
Next, in the left hand side of the above equation we use  $2\sqrt {ab}\leq (a+b)$ to get
\begin{align*}
 2\sqrt{\alpha\delta(\lambda)\beta(2-q)(4-q)}\|u\|^3&\leq (2-q)\alpha\delta(\lambda)\|u\|^2
+(4-q)\beta\|u\|^4\\
 &<(p-q)\int_{\mathbb R^N}|x|^{-pb}|u|^{p} dx\leq (p-q)  S^{\frac{-p}{2}}\|u\|^{p}\,,
\end{align*}
where the last estimate comes from \eqref{CKN}, which implies that $\|u\|^{p-3}>C$. Hence $\|u\|\geq \rho$ for some $\rho>0$\,.
 \end{proof}
\begin{cor}\label{nlclosedi}
$N_{\lambda, \mu}^-$ is a closed set in the $D^{1,2}_a(\mathbb R^N)$ topology.
\end{cor}
\begin{proof}
Let $\{u_k\}$ be a sequence in $N_{\lambda, \mu}^-$ such that $u_k\rightarrow u$ in $D^{1,2}_a(\mathbb R^N)$. Then $u\in \overline {N_{\lambda, \mu}^-}=N_{\lambda, \mu}^-\cup\{0\}$, and using Lemma \ref{nlwi} we get $\|u\|=\displaystyle \lim_{k\rightarrow \infty}\|u_k\|\geq \rho>0$. Hence $u\neq 0$ and, therefore, $u\in N_{\lambda, \mu}^-$.
\end{proof}
\section{{Existence of  Palais-Smale sequence}}
{In this section, our aim is to extract the minimizing Palais-Smale sequences out of Nehari decompositions.}
 Let us fix $u\in {N}_{\lambda, \mu}$ and define $\mathcal{F}_u:\mathbb{R}\times D^{1, 2}_a(\mathbb R^N)\rightarrow \mathbb{R}$ as follows:
\begin{align*}
  \mathcal{F}_u(t,v) &=\langle J_{\lambda, \mu}^\prime (t(u-v)), (u-v)\rangle_* \\
  &={t^{2}\left(\alpha\|u-v\|^{2}-\lambda\int_{\mathbb R^N}h(x)|x|^{-2(1+a)}|u-v|^2dx\right)} +t^{4}\beta\|u-v\|^{4}-t^{q}\mu \int_{\mathbb R^N}{f(x)|u-v|^{q}dx}\\&\quad -t^{p}\int_{\mathbb R^N}|x|^{-pb}|u-v|^{p} dx\,.
\end{align*}
Then $\mathcal{F}_u(1,0) = 0,\; \frac{\partial}{ \partial t}\mathcal{F}_u(1,0)\neq 0$, since ${N}_{\lambda, \mu}^{0} = \emptyset$ for $\lambda\in (0, \alpha \lambda_1(h))$ and $\mu \in (0, \mu_1)$. So, we can apply the implicit function theorem to obtain a differentiable function $\xi : \mathcal{B}(0, \epsilon) \subseteq D^{1, 2}_a(\mathbb R^N) \rightarrow \mathbb{R}$ such that $\xi(0) = 1$,
\begin{equation}\label{eqzi}
\langle\xi^{'}(0), v\rangle=\frac{J_{\lambda, \mu}^{\prime\prime}(u)(u, v)+\langle J_{\lambda, \mu}^\prime (u), v\rangle_*}{J_{\lambda, \mu}^{\prime\prime}(u)(u, u)}\,,
\end{equation}
and $\mathcal{F}_u(\xi(v),v) = 0$ for all $v\in \mathcal{B}(0, \epsilon)$. Therefore, $ \xi(v)(u-v) \in {N}_{\lambda, \mu}$. For easy reference, we  summarize the above discussion in the following lemma.

\begin{lem}\label{ziii}
For a given $u \in {N}_{\lambda, \mu}^+$ (or $u\in {N}_{\lambda, \mu}^-$ ) and $\lambda \in (0, \alpha \lambda_{1}(h)), \mu \in (0, \mu_1)$ there exists $\epsilon > 0$ and a differentiable function
$\xi : \mathcal{B}(0,\epsilon) \subseteq D^{1,2}_a(\mathbb R^N) \rightarrow \mathbb{R}$ such that $\xi(0)=1,$ $\xi(v)(u-v)\in {N}_{\lambda, \mu}^+$ (or $\xi(v)(u-v)\in {N}_{\lambda, \mu}^-$)
and \eqref{eqzi} holds.
\end{lem}
{
\begin{rem}
	Note that  the Nehari manifold lacks the basic linear space structure in order to give meaning to the derivative of the functional $J_{\lambda, \mu}$ restricted to $N_{\lambda, \mu}$. The above lemma is crucial to give sense to it by constructiong a variation of a point lying in Nehari manifold.
	\end{rem}}
Next, using the Lemma \ref{ziii}, we prove the following proposition which shows the existence of a Palais-Smale sequence.

\begin{pro}\label{prp1}
Let $\lambda \in (0,\alpha \lambda_{1}(h))$. Then:
\begin{enumerate}
\item For $\mu\in (0, \mu_1)$, where we recall that $\mu_1$ is defined in Lemma \ref{l22.2i}, there exists a minimizing sequence $\{u_k\} \subset {N}_{\lambda, \mu}$   such that
\begin{center}
    $J_{\lambda, \mu}(u_{k}) = \theta_{\lambda, \mu}+o_k(1)$\ and\ $J_{\lambda, \mu}^{'}(u_{k}) = o_k(1)\,;$
\end{center}
\item There exists $\mu_3>0$ such that, if $\mu\in (0, \mu_3)$, then there exists a minimizing sequence $\{u_k\} \subset {N}_{\lambda, \mu}^-$   such that
\begin{center}
    $J_{\lambda, \mu}(u_{k}) = \theta_{\lambda, \mu}^-+o_k(1)$\ and\ $J_{\lambda, \mu}^{'}(u_{k}) = o_k(1).$
\end{center}
\end{enumerate}
\end{pro}
\begin{proof}
To avoid any repetition, we only prove the part ${\it (1)}$ of the above Proposition. The proof for the part ${\it (2)}$ is similar. From Lemma \ref{cbbi}, $J_{\lambda, \mu}$ is bounded below on $N_{\lambda, \mu}$. So, by Ekeland Variational Principle, there exists a minimizing sequence $\{u_k\}\in N_{\lambda, \mu}$ such that
\begin{align}
J_{\lambda, \mu}(u_k)&\leq \theta_{\lambda, \mu}+\frac{1}{k},\label{Pek1}\\
J_{\lambda, \mu}(v)&\geq J_{\lambda, \mu}(u_k)- \frac{1}{k}\|v-u_k\|\;\;\text{for all}\;\;v\in  N_{\lambda, \mu}.\nonumber
\end{align}
Using \eqref{Pek1} and Lemma \ref{3a}, it is easy to show that $u_k\not\equiv 0$. Indeed, using \eqref{Pek1} and  Lemma \ref{3a}, we get
\begin{equation}\label{psre12}
\frac{\theta_{\lambda, \mu}}{2}\geq  J_{\lambda, \mu}(u_{k})\geq \left(\frac{1}{2}-\frac{1}{p}\right)\left(\alpha\|u_k\|^{2}-\lambda\int_{\mathbb R^N}h(x)|x|^{-2(1+a)}|u_k|^2dx\right)-\mu\left(\frac{1}{q}-\frac{1}{p}\right)\int_{\mathbb R^N }  f(x)|u_k|^{q}dx
\end{equation}
which implies (using $\theta_{\lambda, \mu}<0$)
\begin{equation}\label{psre13}
\begin{aligned}
&-\frac{pq}{2(p-q)}\theta_{\lambda, \mu}\leq \mu\int_{\mathbb R^N }  f(x)|u_k|^{q}dx\leq \mu \|f\|_oS^{\frac{-q}{2}}\|u_k\|^q.
\end{aligned}
\end{equation}
From \eqref{psre13}, it immediately follows that
\[
\|u_k\|\geq \left(\frac{-pq \theta_{\lambda, \mu} S^{\frac{q}{2}}}{2\mu(p-q)\|f\|_o}\right)^\frac{1}{q}.
\]
Moreover, combining \eqref{psre12} and $\theta_{\lambda, \mu}<0$, we obtain
\[
\|u_k\|\leq \left(\frac{2\mu (p-q)\|f\|_oS^{\frac{-q}{2}}}{q(p-2)}\right)^\frac{1}{2-q}.
\]
Hence the first result in part ${\it (1)}$ is proved. Next we claim that $\|J_{\lambda, \mu}^\prime(u_k)\|\rightarrow 0$ as $k \rightarrow 0$.
Indeed, from Lemma \ref{ziii} we get a differentiable functions $\xi_k:\mathcal{B}(0, \epsilon_k)\rightarrow \mathbb{R}$ for some $\epsilon_k>0$ such that $\xi_k(v)(u_k-v)\in {N}_{\lambda, \mu}$\;\; $\textrm{for all}\;\; v\in \mathcal{B}(0, \epsilon_k).$
Now, with fixed $k$, choose $0<\rho<\epsilon_k$ and let $0\neq u\in D^{1,2}_a(\mathbb R^N)$. Then, let $v_\rho={\rho u}/{\|u\|}$ and set $\eta_\rho=\xi_k(v_\rho)(u_k-v_\rho)$. Since $\eta_\rho \in  {N}_{\lambda, \mu}$, we get from \eqref{Pek1} that
\begin{align*}
J_{\lambda, \mu}(\eta_\rho)-J_{\lambda, \mu}(u_k)\geq-\frac{1}{k}\|\eta_\rho-u_k\|.
\end{align*}
Next, the Mean Value Theorem yields
\begin{equation*}
\langle J_{\lambda, \mu}^{\prime}(u_k),\eta_\rho-u_k\rangle_*+o_k(\|\eta_\rho-u_k\|)\geq -\frac{1}{k}\|\eta_\rho-u_k\|.
\end{equation*}
Hence
\begin{align*}
\langle J_{\lambda, \mu}^{\prime}(u_k),-v_\rho\rangle_*+(\xi_k(v_\rho)-1)\langle  J_{\lambda, \mu}^{\prime}(u_k),(u_k-v_\rho)\rangle_* \geq -\frac{1}{k}\|\eta_\rho -u_k\|+o_k(\|\eta_\rho -u_k\|)
\end{align*}
and, since $\langle J_{\lambda, \mu}^{\prime}(\eta_\rho),(u_k-v_\rho)\rangle_*=0$, we have
\begin{align*}
-\rho{\left \langle J_{\lambda, \mu}^{\prime}(u_k),\frac{u}{\|u\|}\right\rangle_*}&+(\xi_k(v_\rho)-1)\langle J_{\lambda, \mu}^{\prime}(u_k)-J_{\lambda, \mu}^{\prime}(\eta_\rho),(u_k-v_\rho)\rangle_* \geq -\frac{1}{k}\|\eta_\rho-u_k\|+o_k(\|\eta_\rho-u_k\|).
\end{align*}
Therefore,
\begin{align}\label{fifte}
{\left \langle J_{\lambda, \mu}^{\prime}(u_k),\frac{u}{\|u\|}\right\rangle_*} \leq \frac{1}{k\rho}\|\eta_\rho-u_k\|&+\frac{o_k(\|\eta_\rho-u_k\|)}{\rho}+
\frac{(\xi_k(v_\rho)-1)}{\rho}\langle  J_{\lambda, \mu}^{\prime}(u_k)-J_{\lambda, \mu}^{\prime}(\eta_\rho),(u_k-v_\rho)\rangle_*.
\end{align}
And, since
$\displaystyle
\|\eta_\rho-u_k\|\leq \rho|\xi_k(v_\rho)|+|\xi_k(v_\rho)-1|\|u_k\|
$
and
$$
\displaystyle\lim_{\rho\rightarrow 0^+}\frac{|\xi_k(v_\rho)-1|}{\rho}\leq \|\xi_k'(0)\|_*,
$$
we obtain, by letting  $\rho\rightarrow 0^+$\ in \eqref{fifte}, that
\begin{equation*}
{\left\langle J_{\lambda, \mu}^\prime(u_k),\frac{u}{\|u\|}\right\rangle_*}\leq\frac{C}{k}(1+{\|\xi_k^{'}(0)\|_*})
\end{equation*}
for some constant $C>0$, independent of $u$.  So, if we can show that $\|\xi_k^{'}(0)\|$ is bounded then we are done.
From Lemma \ref{ziii}, the boundedness of $\{u_k\}$, and H\"older's inequality, we get, for some $K>0$, that
\begin{align*}
{\langle \xi^{\prime}(0),v\rangle_*}= \frac{K\|v\|}{J_{\lambda, \mu}^{\prime\prime}(u)(u, u)}.
\end{align*}
Therefore, it suffices to show that the denominator in the above expression
is bounded away from zero. Suppose not. Then there exists a subsequence, still denoted by $\{u_k\}$, such that
\begin{equation}\label{baw}
(2-q)\left(\alpha\|u_k\|^{2}-\lambda\int_{\mathbb R^N}h(x)|x|^{-2(1+a)}|u_k|^2dx\right) + (4-q)\beta\|u_k\|^{4}-(p-q)\displaystyle \int_{\mathbb R^N}|x|^{-pb}|u_k|^{p} dx=o_k(1).
\end{equation}
From \eqref{baw} and using {$\{u_k\}\subset N_{\lambda, \mu}$}, we get $E_{\lambda, \mu}(u_k)=o_k(1)$. However, using the fact that $\|u_k\|\geq C>0$ and following the proof of Lemma \ref{l22.2i}, we get that $E_{\lambda, \mu} (u_k)>C_1$ for some $C_1>0$ and for all $k$, which is a contradiction. The proof of Proposition \ref{prp1} is now complete.
\end{proof}

\section{Compactness Results}
{As we know that the problem in consideration has two type of compactness issues. First one is because of unbounded domain and the second is because of critical growth. This section is devoted to study these compactness concerns.
\begin{pro}\label{comunbounded}
	Let $(u_{k})$ be a  sequence in $\mathcal{D}_{a}^{1,2}(\mathbb R^N)$ such that $u_k\rightharpoonup u$ in $\mathcal{D}_{a}^{1,2}(\mathbb R^N)$. Then the following convergence holds true
	\begin{align*}
	\lim_{k\to \infty}\int_{\mathbb R^N}h(x)|x|^{-2(1+a)}|u_k|^2dx&=\int_{\mathbb R^N}h(x)|x|^{-2(1+a)}|u|^2dx\\
\lim_{k\to \infty}	\int_{\mathbb R^N}f(x)|u_k|^qdx&=	\int_{\mathbb R^N}f(x)|u|^qdx
	\end{align*}
\end{pro}
\begin{proof}
		Let us define define $K: \mathcal{D}_{a}^{1,2}(\mathbb R^N)\to \mathbb R$ as $$K(u)=	\int_{\mathbb R^N}h(x)|x|^{-2(a+1)}|u|^2dx.$$ Under assumption $\textbf(H)$, the functional is well defined. Indeed, using the continuous inclusion of $\mathcal{D}_{a}^{1,2}(\mathbb R^N) \hookrightarrow L^p_b(\mathbb R^N)$ and H\"older's inequality with exponents $p/2$ and $p/(p-2)$, we get
		\[
			\int_{\mathbb R^N}h(x)|x|^{-2(a+1)}|u|^2dx\leq \left(\int_{\mathbb R^N}|h(x)|^\frac{p}{p-2}|x|^\frac{-p(2-2(b-2))}{p-2} dx\right)^\frac{p-2}{p}\left(\int_{\mathbb R^N} |x|^{-pb}|u_k|^{p} dx \right)^\frac{2}{p}<\infty
		\]
		Next we show that the mapping $u\mapsto K(u)$ is compact. 
For that, take $\{u_k\}\subset \mathcal{D}_{a}^{1,2}(\mathbb R^N) $ such that $u_k\rightharpoonup u$ in $\mathcal{D}_{a}^{1,2}(\mathbb R^N)$. We have to show that $K(u_k)\to K(u)$, upto a subsequences, as $k\to \infty$. Since $u_k\rightharpoonup u$, there exists $C>0$ such that $\|u_k\|\leq C$. 
Now let us estimate the following 
\begin{align*}
\left|K(u_k)-K(u)\right|&\leq \int_{|x|<\rho}|h(x)||x|^{-2(a+1)}\left||u_k|^2-|u|^2\right|dx+\int_{\rho<|x|<R}|h(x)||x|^{-2(a+1)}\left||u_k|^2-|u|^2\right|dx\\
&\quad+\int_{|x|>R}|h(x)||x|^{-2(a+1)}\left||u_k|^2-|u|^2\right|dx:=I_1+I_2+I_3.
\end{align*}
We estimate each of above three estimates one by one. 
Using the fact that $h\in L^\frac{N}{p_0}_{p_0}(\mathbb R^N)$ and the boundedness of the sequence $\{u_k\}$, for any given $\epsilon>0$, there exists $R=R(\epsilon)$ such that 
\[
I_3\leq \|h\|_{L^{N/ {p_0}}_{p_0}(|x|>R)}\|u_k\|^2\leq \frac{\epsilon}{3},
\]
by choosing $ \|h\|_{L^{N/{p_0}}_{p_0}(|x|>R)}\leq \epsilon/(3C^2)$, where $C$ is the bound for the $\mathcal{D}_{a}^{1,2}(\mathbb R^N)$ norm of the sequence $\{u_k\}$. 
Similarly, for $I_1$, for a given $\epsilon>0$ there exists $\rho=\rho(\epsilon)>0$, such that 
\[
I_1\leq \|h\|_{L^{N/ {p_0}}_{p_0}(|x|<\rho)}\|u_k\|^2\leq \frac{\epsilon}{3},
\]
again, by choosing $ \|h\|_{L^{N/{p_0}}_{p_0}(|x|<\rho)}\leq \epsilon/(3C^2)$.
Now, we estimate the integral $I_2$, as follows. Note that we have the following embeddings
\[
\mathcal{D}_{a}^{1,2}(\mathbb R^N)\subset {W}_{a}^{1,2}(B_R\setminus B_\rho)\subset L^\gamma(B_R\setminus B_\rho),
\]
for all $2\leq \gamma\leq 2^*=2N/(N-2)$ with the last embedding being compact for all $2\leq \gamma<2^*$. Keeping in  mind the assumption on $h(x)$, being $h\in L^{N/{p_0}}_{\mathrm {loc}}(\mathbb R^N\setminus \{0\})$ and $p=2N/(N+2(b-a)-2)<2^*$, we get 
\begin{align*}
\int_{\rho<|x|<R}|h(x)||x|^{-2(a+1)}|u_k-u|^2dx&\leq \rho^{-2(a+1)}\int_{\rho<|x|<R}|h(x)||u_k-u|^2dx\\
&\leq \rho^{-2(a+1)}\int_{|x|<R}|h(x)||u_k-u|^2dx\\
&\leq \rho^{-2(a+1)}\|h\|_{L^{N/{p_0}}{(|x|<R)}}\|u_k-u\|_{L^2(|x|<R)}<\frac{\epsilon}{3}.
\end{align*}
for $k>k_0$, for some $k_0=k_0(\epsilon)$ sufficiently large. 
Hence for a given $\epsilon>0$, there exists $k_0=k_0(\epsilon)$ such that 
\[
I_2=\int_{\rho<|x|<R}|h(x)||x|^{-2(a+1)}\left||u_k|^2-|u|^2\right|dx< \frac{\epsilon}{3}
\]
for $k>k_0$.
Now combining above three estimates we get
\[
\left|K(u_k)-K(u)\right|<\epsilon \;\;\text{for }\;\ k>k_0
\]
Hence the proof of the first convergence of the Proposition. \\
Now we prove the second convergence of the proposition as follows. 
	Define $K: \mathcal{D}_{a}^{1,2}(\mathbb R^N)\to \mathbb R$ as $$K(u)=	\int_{\mathbb R^N}f(x)|u|^qdx.$$ It is easy to see from assumption $\textbf(F)$ that $K$ is well defined. In fact, 
	\[
	0\leq K(u)\leq S^\frac{-q}{2}\|f\|_o\|u\|^q<\infty, \;\;\textrm{for all }\;\;u\in\mathcal{D}_{a}^{1,2}(\mathbb R^N).
	\]
Next we claim that $K$ is weakly continuous on $\mathcal{D}_{a}^{1,2}(\mathbb R^N)$, that is, $u_k\rightharpoonup u$ in  $\mathcal{D}_{a}^{1,2}(\mathbb R^N)$ implies $K(u_k)\to K(u)$ in $\mathbb R$. 
	Since $u_k\rightharpoonup u$ in  $\mathcal{D}_{a}^{1,2}(\mathbb R^N)$, we can assume that the sequence $\{u_k\}$ is bounded in $\mathcal{D}_{a}^{1,2}(\mathbb R^N)$, that is, $\|u_k\|\leq C$, uniformly for all $k\in \mathbb N$ and for some positive constant 
	$C$.
	As $|x|^{bq}f(x) \in L^\frac{p}{p-q}(\mathbb R; dx)$, for any given $\epsilon>0$, there exists an open ball $B_R$ arround $0\in \mathbb R^N$ of radius $R=R(\epsilon)$ such that 
	\[
	\int_{\mathbb R^N\setminus {B_R}}|x|^\frac{bpq}{p-q}|f(x)|^\frac{p}{p-q} dx<\frac{S^\frac{q}{2}\epsilon^\frac{p}{p-q}}{4C^q}.
	\]
	Hence from H\"older inequality and above estimate, 
	\begin{equation}\label{estun}
		\int_{\mathbb R^N\setminus {B_R}}f(x)|u_k|^qdx\leq \frac{\epsilon}{4}
	\end{equation}
Similarly, (possibly for different $R$)
\begin{equation}\label{estu}
 \int_{\mathbb R^N\setminus {B_R}}f(x)|u|^qdx\leq \frac{\epsilon}{4}.
\end{equation}
 Combining \eqref{estun} and \eqref{estu}, we get 
\[
\left| \int_{\mathbb R^N\setminus {B_R}}f(x)(|u_k|^q-|u|^q) dx \right|<\frac{\epsilon}{2}.
\] 
Note that we have the following embeddings 
\[
D^{1, 2}_a(\mathbb R^N)\hookrightarrow W^{1, 2}_a(B_R)\hookrightarrow  L^q(B_R),
\]
where the second embedding is compact as $1<q<2^*=2N/(N-2)$.  Using $f\in L^\infty_{\mathrm {loc}}(\mathbb R^N)$ and the compactness of the embedding $W^{1, 2}_a(B_R)\hookrightarrow L^q(B_R)$,  we get the following 
\[
\int_{{B_R}}f(x)|u_k-u|^q dx\leq \sup_{x\in B_R}|f(x)| \int_{{B_R}}|u_k-u|^q dx\leq \frac{\epsilon}{2}, \;\;\text{for }\;\; k \;\;\text{large}.
\]
Hence, for $k$ large, 
\[
0\leq \left| \int_{{B_R}}f(x)(|u_k|^q-|u|^q) dx \right|< \frac{\epsilon}{2}.
\] 
	\end{proof}
}
{In the presence of critical Caffarelli-Kohn-Nirenberg growth, we have the following compactness result for Palais-Smale sequences under certain threshold.}
\begin{pro}\label{crcmpi}
Let $(u_{k})$ be a bounded sequence in $\mathcal{D}_{a}^{1,2}(\mathbb R^N)$ such that
$$
J_{\lambda, \mu}(u_k) \to c<c_{\lambda, \mu} \quad\text{and}\quad
J'_{\lambda, \mu}(u_k) \to 0 \text{ in  } (\mathcal{D}_{a}^{1,2}(\mathbb R^N))^{*}\ \text{as } k\to \infty,
$$
where
\begin{equation}\label{clambdamu}
c_{\lambda, \mu}= \left(\frac{1}{2}-\frac{1}{p}\right)\left({\alpha  S}\right)^\frac{p}{p-2}-\mu^\frac{2}{2-q}\left(\frac{(4-q)\|f\|_* S^\frac{-q}{2}}{4q}\right)^\frac{2}{2-q} \left(\frac{2-q}{2}\right)\left(\frac{2q}{\alpha\delta(\lambda)}\right)^\frac{q}{2-q}.
\end{equation}
Then $\{u_k\}$ possesses a  subsequence that strongly converges in $\mathcal{D}_{a}^{1,2}(\mathbb R^N)$.
\end{pro}

\begin{proof}
Since $(u_k)$ is bounded in $\mathcal{D}_{a}^{1,2}(\mathbb R^N)$ we have, module a subsequence, that
\begin{gather*}
 u_{k}\rightharpoonup u \quad \text{in }  \mathcal{D}_{a}^{1,2}(\mathbb R^N), \\
u_{k}(x)\to u(x) \quad \text{ a.e. in } \mathbb R^N,
\end{gather*}
as $k\to +\infty$.
Moreover, using the concentration-compactness principle due to Lions, we obtain an at most countable index set $I$, sequences $(x_i) \subset \mathbb{R}^N$,\ $(\eta_i), (\nu_i) \subset (0,\infty)$, and finite measures $\eta, \nu$ such that
\begin{equation}\label{CCcon}
|x|^{-2a}|\nabla u_k|^2 \rightharpoonup |x|^{-2a}|\nabla u|^2 + \eta
\quad\text{and}\quad
 |x|^{-pb}|u_k|^{p} \rightharpoonup |x|^{-pb}|u|^{p}+ \nu,
\end{equation}
in the sense of measures, as $k\to +\infty$, where
\begin{equation*}
\nu  =  \sum_{i \in I}\nu_{i}\delta_{x_{i}},\quad
\eta\geq \sum_{i \in I}\eta_{i}\delta_{x_{i}},\quad
{S} \nu_{i}^{2/p}\leq \eta_{i},
\end{equation*}
for all $i \in I$ and $\delta_{x_i}$, the Dirac mass at $x_i$.
Our goal is to show that $I$ is empty. Suppose not. Then, for any $i\in I$, we can consider smooth cut-off functions $\psi_{\epsilon, i}(x)$, centered at $x_i$, such that $0\leq \psi_{\epsilon, i}(x)\leq 1$, $\psi_{\epsilon, i}(x)=1$ in $B_{\frac{\epsilon}{2}}(x_i)$, $\psi_{\epsilon, i}(x)=0$ in $B^c_{\epsilon}(x_i)$,  and $|\nabla \psi_{\epsilon, i}(x)|\leq {C}/{\epsilon}$. Then $\{u_k\psi_{\epsilon, i}\}$ is bounded in $D^{1, 2}_a(\mathbb R^N)$ and has compact support. By using the boundedness of the sequence, we have that
\begin{align}\label{flim}
\displaystyle \lim_{\epsilon\rightarrow 0}\lim_{k\rightarrow\infty}\int_{\mathbb R^N}f(x)|u_k|^{q}\psi_{\epsilon,i}dx=0.
\end{align}

So, using \eqref{flim}, we estimate {$\langle J_{\lambda, \mu}'(u_k), (u_k\psi_{\epsilon, i})\rangle_*$} as follows:
\begin{align*}
0&=\displaystyle\lim_{\epsilon\rightarrow 0}\displaystyle \lim_{k\rightarrow \infty} \langle J^\prime_{\lambda, \mu}(u_k), \psi_{\epsilon,i}u_k\rangle\\
&=\displaystyle\lim_{\epsilon \rightarrow 0}\displaystyle \lim_{k\rightarrow \infty}\left\{(\alpha+\beta \|u_k\|^2)\int_{\mathbb R^N}|x|^{-2a}\nabla u_k \nabla (\psi_{\epsilon, i}u_k)dx-\lambda\int_{\mathbb R^N}h(x)|x|^{-2(1+a)}|u_k|^2\psi_{\epsilon,i}dx-\int_{\mathbb R^N}|x|^{-pb}|
u_k|^{p}\psi_{\epsilon,i}dx\right\}\,,
\end{align*}
which implies
\begin{align*}
0&=\displaystyle\lim_{\epsilon \rightarrow 0}\displaystyle \lim_{k\rightarrow \infty}\left\{\alpha\int_{\mathbb R^N}|x|^{-2a}|\nabla u_k|^2 \psi_{\epsilon, i}dx+\alpha \int_{\mathbb R^N}|x|^{-2a}u_k\nabla u_k \nabla \psi_{\epsilon, i}dx\right.\\&+\beta\|u_k\|^2 \int_{\mathbb R^N}|x|^{-2a}|\nabla u_k |^2 \psi_{\epsilon, i}dx+
\beta \|u_k\|^2 \int_{\mathbb R^N}|x|^{-2a}u_k \nabla u_k \nabla \psi_{\epsilon, i}dx
\\ &\left.-\lambda\int_{\mathbb R^N}h(x)|x|^{-2(1+a)}|u_k|^2\psi_{\epsilon,i}dx-\int_{\mathbb R^N}|x|^{-pb}|u_k|^{p}\psi_{\epsilon,i}dx\right\}.\\
\end{align*}
Now, using
\begin{align*}
0&\leq\displaystyle\lim_{\epsilon \rightarrow 0}\displaystyle \lim_{k\to\infty} \left|\int_{\mathbb R^N}|x|^{-2a}u_k\nabla u_{k}\nabla \psi_{\epsilon, i} dx \right|\\
&\leq \displaystyle\lim_{\epsilon \rightarrow 0}\displaystyle \lim_{k\to
\infty}\left(\int_{\mathbb R^N}|x|^{-2a}|\nabla u_{k}|^{2}dx\right)^{1/2}
\left(\int_{B_{\epsilon}(x_i)}|x|^{-2a}|\nabla
\psi_{\epsilon, i}|^{2}|u_{k}|^{2}dx\right)^{1/2}\longrightarrow 0\,,
\end{align*}
 we get
 \begin{align*}
0&\geq\displaystyle\lim_{\epsilon \rightarrow 0}\displaystyle \lim_{k\rightarrow \infty}\left\{\alpha\int_{\mathbb R^N}|x|^{-2a}|\nabla u_k|^2 \psi_{\epsilon, i}dx+\beta\|u_k\|^2 \int_{\mathbb R^N}|x|^{-2a}|\nabla u_k |^2 \psi_{\epsilon, i}dx\right.\\&\left.-\lambda\int_{\mathbb R^N}h(x)|x|^{-2(1+a)}|u_k|^2\psi_{\epsilon,i}dx-\int_{\mathbb R^N}|x|^{-pb}|u_k|^{p}\psi_{\epsilon,i}dx\right\}\,.
\end{align*}
Next, using the compact support of $\psi_{\epsilon, i}$ and the local compactness of the sequence  {together with integrability assumption on $h(x)$ in suitable weighted integrable spaces}, we can show that
\[
\displaystyle\lim_{\epsilon \rightarrow 0}\displaystyle \lim_{k\rightarrow \infty} \displaystyle \int_{\mathbb R^N}h(x)|x|^{-2(1+a)}|u_k|^2\psi_{\epsilon,i}dx=0\,.
\]
Hence

\begin{align*}
0&\geq\displaystyle\lim_{\epsilon \rightarrow 0}\displaystyle \lim_{k\rightarrow \infty}\left\{\alpha\int_{\mathbb R^N}|x|^{-2a}|\nabla u_k|^2 \psi_{\epsilon, i}dx-\int_{\mathbb R^N}|x|^{-pb}|u_k|^{p}\psi_{\epsilon,i}dx\right\}\\
&\geq \alpha \eta_i-\nu_i \,,
\end{align*}
which implies $\nu_i\geq {\alpha \eta_i}$. From the relation $\eta_i\geq {S}\nu_i^{\frac{2}{p}}$ it follows that $\eta_i\geq \left({\alpha^2  S^p}\right)^\frac{1}{p-2}$ or $\eta_i=0$. We claim that $\eta_i\geq \left({\alpha^2  S^p}\right)^\frac{1}{p-2}$ is not possible to hold. We argue by contradiction. Suppose
 \begin{align}\label{lereq}
 \eta_i\geq \left({\alpha^2  S^p}\right)^\frac{1}{p-2}.
\end{align}
{Now,
\begin{align*}
c&=\displaystyle\lim_{k\rightarrow \infty} \left\{ J_{\lambda, \mu}(u_k)-\frac{1}{4}\langle  J_{\lambda, \mu}^{\prime}(u_k), u_k\rangle_*\right\}\\&=\lim_{k\to \infty}
\left\{\frac14 \alpha\|u_k\|^2 -\frac{\lambda}{4}\int_{\mathbb R^N}h(x)|x|^{-2(1+a)}|u_k|^2dx +\left(\frac14-\frac1p\right)\int_{\mathbb R^N}|x|^{-pb}|u_k|^pdx-\mu\left(\frac{1}{q}-\frac{1}{4}\right)\int_{\mathbb R^N}f(x)|u_k|^q dx\right\}
\end{align*}
Now using Proposition \ref{comunbounded} and \eqref{CCcon}, we get
\begin{align*}
c=&\geq \frac{1}{4}\alpha\left(\|u\|^2+\sum_{i\in I}\eta_i\delta_{x_i}\right)-\frac{\lambda}{4}\int_{\mathbb R^N}h(x)|x|^{-2(1+a)}|u|^2dx+ \left(\frac{1}{4}-\frac{1}{p}\right)\left(\|u\|_{L^p_b}^{p}+\sum_{i\in I}\nu_i\delta_{x_i}\right)-\mu \left(\frac{1}{q}-\frac{1}{4}\right)\int_{\mathbb R^N}f(x)|u|^qdx\\
\end{align*}
 Using \eqref{lereq} and $\lambda<\alpha \lambda_1(h)$,  we have}
\begin{align*}
c&\geq \frac14 \alpha\eta_{i_0}+\left(\frac{1}{4}-\frac{1}{p}\right)\nu_{i_0}+\frac{1}{4}\left(\alpha\|u\|^2
-{\lambda}\int_{\mathbb R^N}h(x)|x|^{-2(1+a)}|u|^2dx\right)-\mu \left(\frac{1}{q}-\frac{1}{4}\right)\|f\|_o S^{\frac{-q}{2}}\|u\|^q\\
&\geq \frac14 \alpha\eta_{i_0}+\left(\frac{1}{4}-\frac{1}{p}\right)\nu_{i_0}+\frac{1}{4}\alpha\delta(\lambda)\|u\|^2-\mu \left(\frac{1}{q}-\frac{1}{4}\right)\|f\|_o S^{\frac{-q}{2}}\|u\|^q
\\&\geq \left(\frac{1}{2}-\frac{1}{p}\right)\left({\alpha  S}\right)^\frac{p}{p-2}-\mu^\frac{2}{2-q}\left(\frac{(4-q)\|f\|_o S^\frac{-q}{2}}{4q}\right)^\frac{2}{2-q} \left(\frac{2-q}{2}\right)\left(\frac{2q}{\alpha\delta(\lambda)}\right)^\frac{q}{2-q},
\end{align*}
which is a contradiction. Note that the last term in the above inequality is a consequence of the maximum value of the algebraic expression

\begin{equation*}
\frac{1}{4}\alpha\delta(\lambda)t^2-\mu \left(\frac{1}{q}-\frac{1}{4}\right)\|f\|_o S^{\frac{-q}{2}}t^q.
\end{equation*}
Hence $I$ is empty and we can conclude that $$\displaystyle \int_{\mathbb R^N}|x|^{-pb}|u_k|^{p} dx\rightarrow \displaystyle \int_{\mathbb R^N}|x|^{-pb}|u|^{p} dx.$$
Hence the proof follows.

\end{proof}

\section{{Existence of first solution in $N^+_{\lambda, \mu}$: Proof of Theorem \ref{22mht1} (i)}}
Let us fix
 \[
\mu_4= \left(\left(\frac{1}{2}-\frac{1}{p}\right)\left({\alpha S}\right)^\frac{p}{p-2}\left(\frac{(4-q)\|f\|_o S^\frac{-q}{2}}{4q}\right)^\frac{2}{q-2} \left(\frac{2}{2-q}\right)\left(\frac{2q}{\alpha\delta(\lambda)}\right)^\frac{q}{q-2}\right)^\frac{2-q}{2}
\]
so that $c_{\lambda, \mu}>0$. Set $\mu_0=\min\{\mu_1, \mu_2, \mu_4\}$.
Now as the functional is bounded below in $N_{\lambda, \mu}$, we minimize  $J_{\lambda, \mu}$
 in $N_{\lambda, \mu}$ and using Proposition \ref{prp1} (1), we get a minimizing Palais-Smale sequence $\{u_k\}$ such that $J_{\lambda, \mu}(u_k)\rightarrow \theta_{\lambda, \mu}$.  It is  routine to verify that the sequence is bounded in $D^{1, 2}_a(\mathbb R^N).$ Hence we can assume a weak limit $u_0$ of the sequence in $D^{1, 2}_a(\mathbb R^N)$. Now from Lemma \ref{3a} we know that $\theta_{\lambda, \mu}<0$ hence using compactness result as in  Proposition \ref{crcmpi}, we get the minimizer $u_0$ of  $J_{\lambda, \mu}$ in $N_{\lambda, \mu}$ for $\lambda\in(0, \alpha\lambda_1(h))$ and $\mu \in (0, \mu_0)$ with $J_{\lambda, \mu}(u_0)<0$. Now we claim that $u_0 \in N_{\lambda, \mu}^+$ for $\mu\in (0,\mu_0)$.
 If not then $u_0\in N_{\lambda, \mu}^-$ as $ N_{\lambda, \mu}^0=\emptyset$ . Note that using $u_0\in N_{\lambda, \mu}$ and $J_{\lambda, \mu}(u_0)<0$ we get $u_0 \in F^+$. Therefore from fibering map analysis, we get unique positive real numbers $t^-(u_0)>t^+(u_0)>0$ such that $t^-u_0\in  N_{\lambda, \mu}^-$ and $t^+u_0\in N_{\lambda, \mu}^+$. By uniqueness, $t^-$ must be equal to 1 as $u_0\in N_{\lambda, \mu}^-$ by our contrary assumption which implies $t^+<1$. Therefore we can find $t_0\in (t^+, t^-)$ such that
\[
 J_{\lambda, \mu}(t^+u_0)=\displaystyle\min_{0\leq t\leq t^-} J_{\lambda, \mu}(tu_0)<J_{\lambda, \mu}(t_0u_0)< J_{\lambda, \mu}(t^-u_0)=J_{\lambda, \mu}(u_0)=\theta_{\lambda, \mu}
 \]
 which is a contradiction of the fact that $u_0$ is a minimizer of $J_{\lambda, \mu}$ in $N_{\lambda, \mu}$. Hence $u_0\in N_{\lambda, \mu}^+.$ Since $J_{\lambda, \mu}(u)=J_{\lambda, \mu}(|u|)$, we can assume that $u_0\geq0$. Now using the fact that $M(t)>\alpha$ and by classical maximum principle,  we get $u_0>0$.
 \noindent Now the following Lemma shows that $u_0$ is indeed a local minimizer of $J_{\lambda, \mu}$ in $D^{1, 2}_a(\mathbb R^N)$.
 \begin{lem}
The function $u_0 \in N_{\lambda, \mu}^+$ is a local minimum of $J_{\lambda, \mu}$ in $D^{1, 2}_a(\mathbb R^N)$ for $\mu<\mu_0$ and $\lambda \in (0, \alpha\lambda_1(h))$.
\end{lem}
\begin{proof}
Since $u_0 \in N_{\lambda, \mu}^{+}$, we have $t^+(u_0)=1
<t_{\max}(u_0)$. Hence by continuity of $u \mapsto t_{\max}(u)$, given
$\epsilon>0$, there exists $\delta=\delta(\epsilon)>0$ such that $1+\epsilon< t_{\max}(u_0-u)$
for all $\|u\|<\delta$. Also, from Lemma \ref{ziii}, for $\delta_1>0$, we obtain a $C^1$ map $\xi: \mathcal {B}(0,\delta)\rightarrow \mathbb R^+$
such that $\xi(u)(u_0-u)\in  N_{\lambda, \mu}^+$, $\xi(0)=1$. Therefore, for
$0<\delta_2=\min \{\delta, \delta_1\}$ and uniqueness of zeros of fibering map, we have $t^+(u_0-u)=
\xi(u)<1+\epsilon<t_{\max}(u_0-u)$ for all $\|u\|<\delta_2$. Since $t_{\max}(u_0-u)>1$, then for all $\|u\|<\delta_2$,
we obtain $J_{\lambda, \mu}(u_0)\leq J_{\lambda, \mu}(t^+(u_0-u)(u_0-u))\leq J_{\lambda, \mu}(u_0-u)$. This shows that $u_0$ is a local minimizer for
$J_{\lambda, \mu}$ in $D^{1, 2}_a(\mathbb R^N)$.
\end{proof}

\section{Existence of second solution in $ N_{\lambda, \mu}^-$}
{In this section, we study the pivotal estimates on minimizers of Caffarelli-Kohn-Nirenberg inequality which helps to establish an energy estimate in $N_{\lambda, \mu}^-$,
eventually leading us to the existence of second solution in $ N_{\lambda, \mu}^-$.}
It is well known (see \cite{BW,CatF,  MR1223899} ) that the minimizer of the following minimization problem
$$
 S=\displaystyle \inf_{u\in D^{1, 2}_a(\mathbb R^N)\setminus \{0\}}\frac{\int_{\mathbb R^N} |x|^{-2a}|\nabla u|^2dx}{\left(\int_{\mathbb R^N}|x|^{-pb}|u|^p dx\right)^\frac2p}
$$
is given by
$$
U_\epsilon(x)=\frac{\left(2pA\epsilon\right)^\frac{1}{p-2}}{\left(\epsilon + |x|^\frac{(p-2)\gamma}{2}\right)^\frac{2}{p-2}},
$$
where $A=\left(\frac{N-2}{2}-a\right)^2$ and $\gamma=2\sqrt{A}$. 
 Moreover, $U_\epsilon$ is a weak solution of the  the following problem
$$
\mathrm {div}(|x|^{-2a}\nabla U_\epsilon)=|x|^{-pb}|U_\epsilon|^{p-2}U_\epsilon\;\;  \text{in } \mathbb R^N,
$$
and satisfies
$$
\int_{\mathbb R^N} |x|^{-2a}|\nabla U_\epsilon|^2dx=\int_{\mathbb R^N}|x|^{-pb}|U_\epsilon|^p dx.
$$
In the following proposition, we prove some crucial estimates for extremal functions $U_\epsilon$, needed to do the energy level analysis of $J_{\lambda, \mu}$ in $N_{\lambda, \mu}^-$. 
{
\begin{pro}
	Let $\eta: \mathbb R^N\to \mathbb R^+\cup \{0\}$ be a $C_c^\infty(\mathbb R^N)$ function such that $0\leq \eta\leq 1$, $|\nabla \eta|\leq C$ and $\eta=1$ in $|x|<r$, $\eta=0$ for $|x|>2r$. Then we have the following estimates for the  family $\{\eta U_\epsilon\}$:
 \begin{equation}\label{eqstimatescri}
 \begin{aligned}
 \|\eta U_\epsilon\|^2&=\|U_\epsilon\|^2+O(\epsilon^\frac{2}{p-2})= S^\frac{p}{p-2}+O(\epsilon^\frac{2}{p-2})\\
 \|\eta U_\epsilon|^{p}_{L_b^{p}(\mathbb R^N)}&=\|U_\epsilon|^{p}_{L_b^{p}(\mathbb R^N)}+O(\epsilon^\frac{p}{p-2})= S^\frac{p}{p-2}+O(\epsilon^\frac{p}{p-2})\\
 \frac{\|\eta U_\epsilon\|^2}{\left( \int_{\mathbb R^N}|x|^{-pb}|\eta U_{\epsilon}|^{p}dx\right)}&= S^\frac{p}{p-2}+O(\epsilon^\frac{2}{p-2})\\
\|\eta U_\epsilon\|^{p-1}_{L^{p-1}(\mathbb R^N)}&\geq C\epsilon^\frac{1}{p-2},
  \end{aligned}
 \end{equation}
 \end{pro}
\begin{proof}
	The above estimates follow the classical steps of estimation as in Brezis-Nirenberg, \cite{BN1}. We give a brief sketch of the steps. .
	\begin{align*}
	\int_{\mathbb R^N}&|x|^{-2a}|\nabla (\eta U_{\epsilon})|^{2}dx=	\int_{\mathbb R^N}|x|^{-2a}|\eta \nabla U_{\epsilon}dx+U_{\epsilon} \nabla \eta|^{2}dx\\&\leq \int_{\mathbb R^N}|x|^{-2a}\eta^2 |\nabla U_{\epsilon}|^2dx+\int_{\mathbb R^N}|x|^{-2a}U_{\epsilon}^2 |\nabla \eta|^{2}dx+2\int_{\mathbb R^N}|x|^{-2a}\eta|\nabla \eta| U_{\epsilon}|\nabla U_{\epsilon}|dx\\&
	\leq \int_{\mathbb R^N}|x|^{-2a} |\nabla U_{\epsilon}|^2+C^2\int_{\{r<|x|<2r\}}|x|^{-2a}U_{\epsilon}^2 dx+2C\int_{\{r<|x|<2r\}}|x|^{-2a} U_{\epsilon}|\nabla U_{\epsilon}|dx\\&=I_1+I_2+I_3
	\end{align*}
	where $I_1=S^\frac{p}{p-2}$. Now by direct computations one can see that $I_2=O(\epsilon^\frac{2}{p-2})$  as follows.
\begin{align*}
\int_{\{r<|x|<2r\}}|x|^{-2a}U_{\epsilon}^2 dx&\leq r^{-2a}(2pA)^\frac{2}{p-2}\epsilon^\frac{2}{p-2} \int_{\{r<|x|<2r\}} \frac{dx}{\left(\epsilon + |x|^\frac{(p-2)\gamma}{2}\right)^\frac{4}{p-2}}\\
&\leq C(a, N ) \epsilon^\frac{2}{p-2} \int_{r}^{2r} \frac{s^{N-1}ds}{\left(\epsilon + s^\frac{(p-2)\gamma}{2}\right)^\frac{4}{p-2}}\\&\leq C(a, N ) \epsilon^\frac{2}{p-2} \int_{r}^{2r} s^{N-1-2\gamma} ds\leq C  \epsilon^\frac{2}{p-2}.
\end{align*}
Similarly one can show that $I_3=O(\epsilon^\frac{2}{p-2})$. Now we prove the second estimate of \eqref{eqstimatescri}.
\begin{align*}
\int_{\mathbb R^N}|x|^{-pb}|\eta U_\epsilon|^p dx&=\int_{\{|x|<2r\}}|x|^{-pb}| \eta U_\epsilon|^p dx\\
&=\int_{\{|x|<2r\}}|x|^{-pb}| U_\epsilon|^p dx+\int_{\{|x|<2r\}}|x|^{-pb}(|\eta|^p-1) |U_\epsilon|^p dx
\\&\leq \int_{\mathbb R^N}|x|^{-pb}| U_\epsilon|^p dx+\int_{\{r<|x|<2r\}}|x|^{-pb}(|\eta|^p-1) |U_\epsilon|^p dx:=I_4+I_5
\end{align*}
where $I_4=S^\frac{p}{p-2}$ and 
\begin{align*}
I_5\leq r^{-pb}\int_{\{r<|x|<2r\}}|U_\epsilon|^p&\leq  r^{-pb}(2pA)^\frac{p}{p-2}\epsilon^\frac{p}{p-2} \int_{\{r<|x|<2r\}} |x|^{-p\gamma} dx\leq C \epsilon^\frac{p}{p-2}
\end{align*}
Third estimate is trivial from first and second estimate of \eqref{eqstimatescri} and the last estimate of \eqref{eqstimatescri} follows the similar steps as in  Lemma 3.8 of \cite{MR2911424}.
\end{proof}
}
 Using above estimates we have the following Lemma which is crucial while studying the energy of the functional in $N_{\lambda, \mu}^-$. Let us assume, without loss of generality, that $0\in  f^+$.  We take $\rho>0$ small enough such that $\mathcal B_{2\rho}(0)\subset {f^+}$. Consider a smooth test functions 
$\eta$ such that $0\leq \eta(x)\leq 1$
in ${f^+}$,  $\eta(x)=1$ on $\mathcal B_{\rho}(0)$ and $\eta(x)=0$ on $\mathcal B^c_{2\rho}(0)$. Define $U_{\epsilon,\eta}=\eta U_\epsilon\in D^{1, 2}_a(\mathbb R^N)$. Then we have the following technical result
\begin{lem}\label{II}
Let $u_{0}$ be the local minimum for the functional $J_{\lambda, \mu}$ in $D^{1, 2}_a(\mathbb R^N)$. Then for every $r>0$
there exists $\epsilon_{0} = \epsilon_{0}(r, \eta) > 0$ , $\beta_0>0$ and $\mu_5>0$ s.t.
\begin{equation*}
J_{\lambda, \mu}(u_{0}+r\;U_{\epsilon,\eta}) <c_{\lambda, \mu}, \;\;\textrm{for}\;\;\epsilon \in (0,\epsilon_{0}), \beta\in (0, \beta_0), \lambda\in (0, \lambda_1(h)) \;\text{and}\; \mu\in (0, \mu_5), 
\end{equation*}
where $c_{\lambda, \mu}$ is given in \eqref{clambdamu}.
\end{lem}
\begin{proof}
 Using elementary inequalities we shall estimate the energy from above. We have
\begin{align*}
J_{\lambda, \mu}(u_{0} + r\;U_{\epsilon,\eta})&=\frac{\alpha}{2}\| u_{0} + r\;U_{\epsilon,\eta}\|^{2}-\frac{\lambda}{2} \int_{\mathbb R^N}h(x)|x|^{-2(1+a)}| u_{0} + r\;U_{\epsilon,\eta}|^2 dx + \frac{\beta}{4}\| u_{0} + r\;U_{\epsilon,\eta}\|^{4}\\&\quad 
- \frac{\mu}{q}\int_{\mathbb R^N} f(x)|u_{0}+ r\;U_{\epsilon, \eta}|^{q}dx- \frac{1}{p}\int_{\mathbb R^N} |x|^{-pb} |u_{0} + r\;U_{\epsilon,\eta}|^{{p}}dx\\
&= \frac{\alpha}{2}\| u_{0}\|^{2} + \frac{\alpha}{2}r^{2}\|U_{\epsilon, \eta}\|^{2} + \alpha\;r \langle u_{0}, U_{\epsilon,\eta}\rangle -\frac{\lambda}{2} \int_{\mathbb R^N}h(x)|x|^{-2(1+a)}| u_{0}|^2 dx\\
&\quad -\frac{\lambda r^2}{2} \int_{\mathbb R^N}h(x)|x|^{-2(1+a)}|U_{\epsilon,\eta}|^2 dx-\lambda r \int_{\mathbb R^N}h(x)|x|^{-2(1+a)}u_{0}U_{\epsilon,\eta}dx+ \frac{\beta}{4}\| u_{0}\|^{4}\\&
 \quad + \frac{\beta}{4} r^{4}\|U_{\epsilon, \eta}\|^{4}
 + \;\beta r^{2} \langle u_{0}, U_{\epsilon,\eta}\rangle ^{2}
 + \frac{\beta}{2} r^{2}\|u_{0}\|^{2}\|U_{\epsilon, \eta}\|^{2} + \beta r^{3} \|U_{\epsilon, \eta}\|^{2} \langle u_{0}, U_{\epsilon, \eta}\rangle\\& \quad +\; \beta r \|u_{0}\|^{2}
 \langle u_{0}, U_{\epsilon, \eta}\rangle- \frac{\mu}{q} \int_{\mathbb R^N} f(x) |u_{0} +  r U_{\epsilon, \eta}|^{q}dx - \frac{1}{p} \int_{\mathbb R^N} |x|^{-pb} |u_{0}+rU_{\epsilon, \eta}|^{p}dx.
 \end{align*}
 Using the fact that $u_{0}$ is a solution of problem \eqref{Pc}, we get
 \begin{align*}
 J_{\lambda, \mu}(u_{0} + r\;U_{\epsilon,\eta})&\leq  J_{\lambda, \mu}(u_0) + \frac{r^2}{2}\left(\alpha\|U_{\epsilon, \eta}\|^2 -\lambda\int_{\mathbb R^N}h(x)|x|^{-2(1+a)}|U_{\epsilon,\eta}|^2 dx\right)\\&\quad + \frac{\beta}{4}r^{4}\|U_{\epsilon, \eta}\|^{4} + \beta\;r^{2}\| u_{0}\|^{2}\|U_{\epsilon, \eta}\|^2 +\frac{\beta}{2}r^{2}\| u_{0}\|^{2}\|U_{\epsilon, \eta}\|^2
 + \beta\;r^{3} \|U_{\epsilon, \eta}\|^{3}\| u_{0}\|\\& \quad -\frac{\mu}{q} \left(\int_{\mathbb R^N} f(x)( |u_{0} + r\;U_{\epsilon,\eta}|^{q}-|u_0|^q-qr|u_{0}|^{q-1} U_{\epsilon,\eta} )dx\right)
 \\&\quad - \frac{1}{p}\left(\int_{\mathbb R^N}|x|^{-pb}(|u_0+rU_{\epsilon,\eta}|^{p}-|u_0|^{p}-{p}ru_0^{p-1} U_{\epsilon,\eta})dz\right).
 \end{align*}
 Now, we estimate the sublinear term in the above inequality as follows:
 \begin{align*}
\frac{\mu}{q}\int_{\mathbb R^N} f(x)( |u_{0} &+ r\;U_{\epsilon,\eta}|^{q}-|u_0|^q-qr|u_{0}|^{q-1} U_{\epsilon,\eta} )dx={\mu}\int_{\mathbb R^N} f(x)\left(\displaystyle \int_0^{r\;U_{\epsilon,\eta}} (|u_{0} + s|^{q-1}-|u_0|^{q-1}) ds\right)dx>0.
\end{align*}
Also, using the inequality
 \[
 (m+n)^s-m^s-n^s-sm^{s-1}n\geq C_1m n^{s-1},\;\;\text{for all}\; (m, n)\in [0, \infty)\times [0, \infty)\;\textrm{and}\;s\geq 2\,,
 \]
 for some $C_1\geq0$, we estimate the critical $p$-term as follows:
\begin{align*}
\int_{\mathbb R^N} |x|^{-pb}|u_0+rU_{\epsilon,\eta}|^{p}dx&-\int_{\mathbb R^N} |x|^{-pb}|u_0|^{p}dx-pr\int_{\mathbb R^N} |x|^{-pb}|u_0|^{p-1}U_{\epsilon, \eta}dx\\&\geq r^p\int_{\mathbb R^N} |x|^{-pb} |U_{\epsilon,\eta}|^{p}dx+C_1r^{p-1}\int_{\mathbb R^N}|x|^{-pb} u_0|U_{\epsilon,\eta}|^{p-1}dx.
\end{align*}
 Letting $\|u_0\|=R$ and using Young's inequality together with the above estimates, we get
\begin{align*}
 J_{\lambda,\mu}(u_{0} + r\;U_{\epsilon,\eta})&\leq  J_{\lambda, \mu}(u_0)+\frac{r^2}{2}\alpha\|U_{\epsilon, \eta}\|^2+\frac{\beta}{4}r^{4}\|U_{\epsilon, \eta}\|^{4}+\frac{3\beta}{2}r^{2}R^{2}\|U_{\epsilon, \eta}\|^2\\&+ \beta R\;r^{3} \|U_{\epsilon,
 \eta}\|^{3}-\frac{1}{p}r^{p}\int_{\mathbb R^N}|x|^{-pb}|U_{\epsilon,\eta}|^{p}dx-C r^{p-1}\int_{\mathbb R^N} |x|^{-pb} u_{0} |U_{\epsilon,\eta}|^{p-1}dx.\\
 &\leq J_{\lambda, \mu}(u_0)+\frac{r^2}{2}\alpha\|U_{\epsilon, \eta}\|^2 +\frac{7\beta}{4}r^{4}\|U_{\epsilon, \eta}\|^{4}-\frac{1}{p}r^{p}\int_{\mathbb R^N}|x|^{-pb}|U_{\epsilon,\eta}|^{p}dx
 \\&-C r^{p-1}\int_{\mathbb R^N} |x|^{-pb} u_{0} |U_{\epsilon,\eta}|^{p-1}dx+\frac{7\beta}{2}R^4.
 \end{align*}
 	{Next, we denote, keeping in mind that $u_0\in C^{1, \gamma}_{\mathrm {loc}}(\mathbb R^N\setminus\{0\})$ by standard elliptic regularity theory, (see for more regularity results of this class of local problems in \cite{GBLI})}
 \begin{align*}
 g(t)&=\frac{t^2}{2}\alpha\|U_{\epsilon, \eta}\|^2+\frac{7\beta}{4}t^{4}\|U_{\epsilon, \eta}\|^{4}-
 {C t^{p-1}\displaystyle \int_{\mathbb R^N} |U_{\epsilon,\eta}|^{p-1}dx}-\frac{1}{p}t^{p}\displaystyle \int_{\mathbb R^N}|x|^{-pb}|U_{\epsilon,\eta}|^{p}dx.
 \end{align*}
 We claim the following \\
 \textbf{Claim:} There exists $t_\epsilon>0$ and $t_1, t_2>0$ (independent of $U_{\epsilon, \eta}$) such that
 \begin{align*}
 g(t_\epsilon )=\displaystyle \sup_{t\geq 0}g(t)\;\textrm{and}\; \frac{d}{dt}g(t)\mid_{t=t_\epsilon}=0
 \end{align*}
 and  $0<t_1\leq t_\epsilon\leq t_2<\infty$.\\
  Since $\displaystyle \lim_{t\rightarrow \infty} g(t)=-\infty$ and $\displaystyle \lim_{t\rightarrow 0^+} g(t)>0$, 
 there exists $t_\epsilon>0$ such that
 \begin{align}\label{mm}
 g(t_\epsilon )=\displaystyle \sup_{t\geq 0}g(t)\;\textrm{and}\; \frac{d}{dt}g(t)\mid_{t=t_\epsilon}=0.
 \end{align}
 From \eqref{mm} we get 
 \begin{align}\label{der0}
 t\alpha\|U_{\epsilon, \eta}\|^2 +{7\beta}t^{3}\|U_{\epsilon, \eta}\|^{4}&=
 {C t^{p-2}\displaystyle \int_{\mathbb R^N} |U_{\epsilon,\eta}|^{p-1}dx}+t^{p-1}\displaystyle \int_{\mathbb R^N}|x|^{-pb}|U_{\epsilon,\eta}|^{p}dx
 \end{align}
 and
 \begin{align}\label{ddn}
 \alpha\|U_{\epsilon, \eta}\|^2&<
 {C t^{p-3}\displaystyle \int_{\mathbb R^N} |U_{\epsilon,\eta}|^{p-1}dx}+C_1t^{p-2}\displaystyle \int_{\mathbb R^N}|x|^{-pb}|U_{\epsilon,\eta}|^{p}dx-{C_3\beta}t^{2}\|U_{\epsilon, \eta}\|^{4}.
\end{align}
 From \eqref{ddn}, it is clear that $t_\epsilon$ is bounded below that is there exist constants $t_1>0$, independent of $U_{\epsilon, \eta}$ such that $0<t_1\leq t_\epsilon$. Also, from \eqref{der0}, we have
 \begin{align*}
 \frac{1}{t^2}\alpha\|U_{\epsilon, \eta}\|^2 +{7\beta}\|U_{\epsilon, \eta}\|^{4}&=
 {C t^{p-5}\displaystyle \int_{\mathbb R^N} |U_{\epsilon,\eta}|^{p-1}dx}+t^{p-4}\displaystyle \int_{\mathbb R^N}|x|^{-pb}|U_{\epsilon,\eta}|^{p}dx.
 \end{align*}
 And, since $p>4$, there exists $t_2>0$, independent of $U_{\epsilon,\eta}$ such that $t_\epsilon\leq t_2<\infty$. This proves the above {claim}.\qed\\

Now taking $\beta=\varepsilon$ and using the above  estimates together with $J_{\lambda, \mu}(u_0)<0$ we get
 \begin{align*}
 J_{\lambda,\mu}(u_{0} + r\;U_{\epsilon,\eta})&\leq  J_{\lambda, \mu}(u_0)+\displaystyle \sup_{t\geq 0}g(t)+\frac{7\varepsilon}{2}R^4\\
 &\leq \sup_{t\geq 0} \displaystyle\left(t^2\alpha\|U_{\epsilon, \eta}\|^2
 -\frac{1}{p}t^{p}\int_{\mathbb R^N}|x|^{-pb}|U_{\epsilon, \eta}|^{p}dx\right) +\frac{5\varepsilon}{4}t_2^{4}\|U_{\epsilon, \eta}\|^{4}-{C t_1^{p-1}\int_{\mathbb R^N} |U_{\epsilon,\eta}|^{p-1}dx}+\frac{7\varepsilon}{2}R^4\\
 &\leq \left(\frac{1}{2}-\frac{1}{p}\right)(\alpha S)^\frac{p}{p-2}+C_2{\epsilon}+O(\epsilon^\frac{2}{p-2})+O(\epsilon^\frac{p}{p-2})-C\epsilon^\frac{1}{p-2}\\
 &\leq \left(\frac{1}{2}-\frac{1}{p}\right)(\alpha S)^\frac{p}{p-2}+C_3 \epsilon^\frac{2}{p-2}-C\epsilon^\frac{1}{p-2}
  \end{align*}
where $C_3, C>0$ are positive constants independent of $\epsilon, \mu$. Next choose $\varepsilon_0>0$ sufficiently small such that, for $\epsilon \in (0, \epsilon_0)$,
\[
C_3 \epsilon^\frac{2}{p-2}-C\epsilon^\frac{1}{p-2}<0.
\]
{Then, there exists some $\mu_5>0$ such that for  $\mu\in (0, \mu_5)$, we have 
$$
C_3 \epsilon^\frac{2}{p-2}-C\epsilon^\frac{1}{p-2}\leq -\mu^\frac{2}{2-q}\left(\frac{(4-q)\|f\|_o S^\frac{-q}{2}}{4q}\right)^\frac{2}{2-q} \left(\frac{2-q}{2}\right)\left(\frac{2q}{\alpha\delta(\lambda)}\right)^\frac{q}{2-q}
$$
and hence  $ J_{\lambda, \mu}(u_{0} + r\;U_{\epsilon,\eta}) \leq c_{\lambda, \mu}$ for $\epsilon \in (0, \epsilon_0)$ and $\mu\in (0, \mu_5)$. This proves the Proposition.}
\end{proof}
Consider the following
\begin{align*}
  U_{1} &= \left\{u\in D^{1, 2}_a(\mathbb R^N)\setminus\{0\}\;\; \big{|} \frac{1}{\|u\|}t^{-}\left(\frac{u}{\|u\|}\right) > 1\right\} \cup \{0\}, \\
  U_{2} &= \left\{u \in D^{1, 2}_a(\mathbb R^N)\{0\}\;\;\big{|} \frac{1}{\|u\|}t^{-}\left(\frac{u}{\|u\|}\right) < 1\right\}.
\end{align*}
Then $N_{\lambda, \mu}^{-}$ disconnects $ D^{1, 2}_a(\mathbb R^N)$ in two connected components $U_{1}$ and $U_{2}$
and $ D^{1, 2}_a(\mathbb R^N)\setminus N_{\lambda, \mu}^{-} = U_{1} \cup U_{2}.$ For each $u \in  N_{\lambda, \mu}^{+},$ we have $1< t_{\max}(u) < t^{-}(u).$
 Since $t^{-}(u) = \frac{1}{\|u\|}t^{-}\left(\frac{u}{\|u\|}\right),$ then $N_{\lambda, \mu}^{+} \subset U_{1}.$
In particular, $u_0 \in U_{1}.$  Now we prove the following lemma
\begin{lem}\label{intersectnlam}
There exists $l_0>0$ such that $u_{0} + l_{0}U_{\epsilon, \eta} \in U_{2}$.
\end{lem}
\begin{proof}
First, we find a constant $c>0$ such that $0 < t^{-} \left(\frac{u_{0}+l\;U_{\epsilon,\eta}}{\|u_{0}+l\;U_{\epsilon,\eta}\|}\right)<c.$ Otherwise, there exists a sequence $\{l_{n}\}$ such that $l_{n} \rightarrow \infty$ and $t^{-}\left(\frac{u_{0}+l_{n}\;U_{\epsilon,\eta}}{\|u_{0}+l_{n}\;U_{\epsilon,\eta}\|}\right) \rightarrow \infty$ as $n \rightarrow\infty.$ Let $v_{n} = \frac{u_{0}+l_{n}\;U_{\epsilon,\eta}}{\|u_{0}+l_{n}\;U_{\epsilon,\eta}\|}.$ Since $t^{-}(v_{n})v_{n} \in N_{\lambda, \mu}^{-} \subset N_{\lambda, \mu}$ and by the Lebesgue dominated convergence theorem,
\begin{align*}
 \int_{\mathbb R^N}|x|^{-pb}|v_n|^p dx &= \frac{1}{\|u_{0}+l_{n}\;U_{\epsilon,\eta}\|^{p}}\int_{\mathbb R^N} |x|^{-pb}|u_{0}+l_{n}\;U_{\epsilon,\eta}|^{p}dx \\
   &= \frac{1}{\|\frac{u_{0}}{l_{n}}+\;U_{\epsilon,\eta}\|^{p}}\int_{\mathbb R^N} |x|^{-pb}\left|\frac{u_{0}}{l_{n}}+U_{\epsilon,\eta}\right|^{p} dx\rightarrow \frac{\displaystyle \int_{\mathbb R^N} |x|^{-pb}|U_{\epsilon, \eta}|^{p}dx}{\|U_{\epsilon, \eta}\|^{p}}\;\; \textrm{as}\;\; n \rightarrow \infty.
\end{align*}
Hence
\begin{align*}
  {J}_{\lambda, \mu}(t^{-}(v_{n})v_{n}) &= \frac{1}{2}(t^{-}(v_n))^2\left(\alpha\|v_n\|^2-\lambda \int_{\mathbb R^N}h(x)|x|^{-2(1+a)}v_n^2 dx\right)+\frac{1}{4}\beta(t^{-}(v_n))^4\|v_n\|^4\\&-\frac{(t^{-}(v_{n}))^{q}}{q}\mu\int_{\mathbb R^N}f(x) |v_{n}|^q\;dx - \frac{(t^{-}(v_{n}))^{p}}{p}\int_{\mathbb R^N} |x|^{-pb}|v_{n}|^{p}dx \rightarrow - \infty\;\; \textrm{as}\;\; n \rightarrow \infty.
\end{align*}
This contradicts that $J_{\lambda, \mu}$ is bounded below on $N_{\lambda, \mu}.$ Let
\begin{equation*}
    l_{0} = \frac{|c^{2}-\|u_{0}\|^{2}|^{\frac{1}{2}}}{\|U_{\epsilon,\eta}\|} + 1,
\end{equation*}
then
\begin{align*}
  \|u_{0}+l_{0}U_{\epsilon,\eta}\|^{2} &= \|u_{0}\|^{2} + (l_{0})^{2}\|U_{\epsilon,\eta}\|^{2}+2l_{0}\langle u_{0}, U_{\epsilon,\eta}\rangle >\|u_{0}\|^{2} + |c^{2}-\|u_{0}\|^{2}| + 2l_{0}\langle u_{0}, U_{\epsilon,\eta}\rangle\\
   &>c^{2}>\left(t^{-}\left(\frac{u_{0}+l_{0}U_{\epsilon,\eta}}{\|u_{0}+l_{0}U_{\epsilon,\eta}\|}\right)\right)^{2}
\end{align*}
that is $u_{0}+l_{0}U_{\epsilon,\eta} \in U_{2}.$
\end{proof}

\noindent \textbf{Proof of Theorem \ref{22mht1}(ii):}
 \noindent {Let us set
 $\mu_{00}=\displaystyle \min\{\mu_1, \mu_2, \mu_3, \mu_5\}$, where $\mu_1, \mu_2$ are defined in \eqref{muone}, \eqref{mutwo} and $\mu_3$ , $\mu_5$ are defined in Proposition \ref{prp1} and Lemma \ref{II} respectively. }  Then using  Lemma \ref{intersectnlam} and the fact that $u_0\in N_{\lambda, \mu}^+\subset U_1$, one can define a continuous path $\gamma(t) = u_0+t\;l_{0}U_{\epsilon,\eta}$ connecting $u_0$ and $u_0+l_{0}U_{\epsilon,\eta}$ from $U_1$ to $U_2$. Then there exists $t_{0} \in (0,1)$
  such that $\gamma_{0}(t_{0})=u_0+t_0\;l_{0}U_{\epsilon,\eta} \in N_{\lambda, \mu}^{-}$ as $N_{\lambda, \mu}^{-}$ disconnects $U_1$ and $U_2$. Therefore, by Lemma \ref{II},
    $\theta_{\lambda, \mu}^{-} \leq J_{\lambda,\mu}(u_0+t_0\;l_{0}U_{\epsilon,\eta})< c_{\lambda, \mu}$ {for $0<\mu<\mu_{00}$.}
 Now from Proposition \ref{prp1} (2), there exists a bounded  minimizing  Palais-Smale sequence
$\{u_{k}\}$ for $J_{\lambda, \mu}$ in $ N_{\lambda, \mu}^{-}$. Since
$\theta_{\lambda, \mu}^{-} <c_{\lambda, \mu}$ by Proposition \ref{crcmpi}, there exists a subsequence $\{u_k\}$ and $u_1$ in  $D^{1, 2}_a(\mathbb R^N)$ such that $u_{k} \rightarrow u_1$ strongly in $ D^{1, 2}_a(\mathbb R^N).$ Now using Corollary \ref{nlclosedi}, $u_1 \in N_{\lambda, \mu}^{-}$ and $J_{\lambda, \mu}(u_{k}) \rightarrow J_{\lambda, \mu}(u_1) = \theta_{\lambda, \mu}^{-}\; \textrm{as}\; k \rightarrow \infty.$ Therefore $u_1$ is also a solution. Moreover, using a similar argument as in the case of first solution, one can show that $u_1$ is a positive solution of the problem $\eqref{Pc}$. Since $N_{\lambda, \mu} ^+\cap N_{\lambda, \mu}^-=\emptyset$, $u_0$ and $u_1$ are distinct. This proves Theorem \ref{22mht1}.
\section*{Acknowledgement} 
\noindent Research supported in part by INCTmat/MCT/Brazil, CNPq and CAPES/Brazil

\end{document}